\documentclass{article}
\usepackage{amsmath, amssymb, amsthm, graphicx, amscd, amsfonts, amscd}

\usepackage[pagewise, displaymath, mathlines]{lineno}

\usepackage{setspace}
\newtheorem{theorem}{Theorem}[section]
\newtheorem{lemma}[theorem]{Lemma}

\theoremstyle{remark}

\def\bb #1{ {\mathbb #1} }
\def\c #1{ {\mathcal #1} }
\def\f #1{ {\mathfrak #1} }
\def\b #1{ {\bf #1} }

\newenvironment{psmallmatrix}
  {\left(\begin{smallmatrix}}
  {\end{smallmatrix}\right)}

\begin{document}
\title{$p$-adic variation of unit root $L$-functions}

\author{C. Douglas Haessig \and Steven Sperber}

\date{\today}
\maketitle

\begin{abstract}
Dwork's conjecture, now proven by Wan \cite{Wan-Higherrankcase-2000, Wan-Rankonecase-2000, Wan-DworkConjectureunit-1999}, states that unit root $L$-functions ``coming from geometry'' are $p$-adic meromorphic. In this paper we study the $p$-adic variation of a family of unit root $L$-functions coming from a suitable family of toric exponential sums. In this setting, we find that the unit root $L$-functions each have a unique $p$-adic unit root. We then study the variation of this unit root over the family of unit root $L$-functions. Surprisingly, we find that this unit root behaves similarly to the classical case of families of exponential sums, as studied in \cite{MR2966711}. That is, the unit root is essentially a ratio of $\c A$-hypergeometric functions.
\end{abstract}


\section{Introduction}

Dwork conjectured \cite{Dwork-NormalizedPeriodsII} that certain $L$-functions, constructed as Euler products of $p$-adic unit roots coming from  the fibers of an algebraic family of $L$-functions, are $p$-adic meromorphic. He proved this in a few cases using the idea of an excellent lifting of Frobenius, but was unable to prove it in general, mainly because excellent lifting in its original form does not always exist. In a series of papers \cite{Wan-Higherrankcase-2000, Wan-Rankonecase-2000, Wan-DworkConjectureunit-1999}, Wan proved Dwork's conjecture using a new technique which avoided excellent lifting. In this paper, we use Wan's techniques, as established in \cite{MR3249829}, to study the $p$-adic variation of unit root $L$-functions.

To solidify concepts, we first consider an example of a unit root $L$-function coming from a family of toric exponential sums. Let $\Psi$ be a nontrivial additive character on $\bb F_q$. Let $f \in \bb F_q[\lambda_1^\pm, \ldots, \lambda_s^\pm, x_1^\pm, \ldots, x_n^\pm]$ be a Laurent polynomial, and consider for each $\bar \lambda \in (\overline{\bb F}_q^\times)^s$ and $m \geq 1$, the exponential sum
\[
S_m(f, \bar \lambda) := \sum_{\bar x \in (\bb F_{q^{m \cdot deg(\bar \lambda)}}^\times)^n} \Psi\circ Tr_{\bb F_{q^{m \cdot deg(\bar \lambda)}} / \bb F_q}(f(\bar \lambda, \bar x)).
\]
Define by $L(f, \bar \lambda, T) := \exp( \sum_{m \geq 1} S_m(f, \bar \lambda) \frac{T^m}{m})$ the associated $L$-function. It is known that $L(f, \bar \lambda, T)^{(-1)^{n+1}}$ is a rational function with a unique $p$-adic unit root, say $\pi_0(\bar \lambda)$, which is also a 1-unit. The unit root $L$-function of this family is defined by
\[
L_\text{unit}(\kappa, T) := \prod_{\bar \lambda \in |\bb G_m^s / \bb F_q |} \frac{1}{1 - \pi_0(\bar \lambda)^\kappa T^{deg(\bar \lambda)}},
\]
where $\kappa$ takes on values in the $p$-adic integers $\bb Z_p$. As mentioned above, in this paper we study the $p$-adic variation of unit root $L$-functions such as these. The following setup is similar to that of the above family, but more technical for the following reason. As unit root $L$-functions come from families, and we wish to study a family of unit root $L$-functions, we need to consider a family of families. The role of the variables in the following is: $x$ denotes the space variables, $\lambda$ denotes the parameters of the family, and $t$ denotes the parameters defining the family of families.

Let $\c {A}$ be a finite subset of $\bb {Z}^n$.  We define the Newton polyhedron of $\c {A}$ at $\infty$, denoted $\Delta_{\infty}(\c {A})$, to be the convex closure of $\c {A} \cup {0}$ in $\bb {R}^n$. We make the simplifying hypothesis that every element $u \in \c {A}$ lies on the Newton boundary at $\infty$  of $\Delta_{\infty}(\c {A})$, that is, the union of all faces of $\Delta_{\infty}(\c {A})$ which do not contain the origin. In other language this is the same as the hypothesis that $w(u) = 1$ for all $u \in \c {A}$  where $w$ is the usual polyhedral weight defined by $\Delta_{\infty}(\c {A})$ (see the next section for definition). The generic polynomial $f$ with  $x$-support equal to $\c {A}$ is given by $f(t, x) = \sum t_u x^u  \in \bb F_q[\{ t_u\}_{u \in \c{A}}, x_1^\pm, \ldots, x_n^\pm]$ where $u$ runs over $ \c {A}$ and $\{t_u\}_{u \in \c {A}}$ are  new variables. Let  $\Delta_\infty(f) (= \Delta_{\infty}(\c {A}))$ be the Newton polyhedron at infinity of $f$.  Let $P(\lambda, x) \in \bb F_q[\lambda_1^\pm, \ldots, \lambda_s^\pm, x_1^\pm, \ldots, x_n^\pm]$ be such that the monomials $\lambda^{\gamma}x^v$ in the support of $P(\lambda, x)$ all satisfy $0 < w(v) < 1$.  Such deformations were studied in \cite{MR3239170}. It is convenient to assume the origin is not in the set  $\c {A}$ and if  $\lambda^{\gamma}x^v$ is in the support of $P$, then $v \neq 0$ so that neither $f$ nor $P$ have a constant term (with respect to the $x$-variables). This assumption will be made throughout this work. Let $G(t,\lambda, x):=f(t,x) + P(\lambda,x)$.

We construct a family of $L$-functions as follows. Let $\bar t \in (\overline{\bb F}_q^*)^{|\c {A}|}$, and denote by $deg(\bar t) = [\bb F_q(\bar t) : \bb F_q]$ the degree of $\bar t$, where $\bb F_q(\bar t)$ means we adjoin every coordinate of $\bar t$ to $\bb F_q$. We will often write $d(\bar t)$ for $deg(\bar t)$. For convenience, write $q_{\bar t} := q^{d(\bar t)}$ so that $\bb F_{q_{\bar t}} = \bb F_q(\bar t)$. Next, let $\bar \lambda \in ({\overline{\bb F}_q}^*)^{s}$. Denote by $deg_{\bar t}(\bar \lambda)$ or $d_{\bar t}(\bar \lambda)$ the degree $[\bb F_{q_{\bar t}}(\bar \lambda) : \bb F_{q_{\bar t}}]$; set $q_{\bar t,\bar \lambda} := q_{\bar t}^{d_{\bar t}(\bar \lambda)}$ and  $\bb F_{q_{\bar t, \bar \lambda}} = \bb F_{q_{\bar t}}(\bar \lambda)$. For each $m \geq 1$, define the exponential sum
\[
S_m(\bar t, \bar \lambda) := \sum_{\bar x \in (\bb F_{{q^m_{\bar t, \bar \lambda}}}^*)^n}  \Psi \circ Tr_{\bb F_{q_{\bar t, \bar \lambda}^m} / \bb F_q}( G(\bar t, \bar \lambda, \bar x))
\]
and its associated $L$-function
\[
L(\bar t, \bar \lambda, T) := \exp\left( \sum_{m=1}^\infty S_m(\bar t, \bar \lambda) \frac{T^m}{m} \right).
\]
It is well-known \cite{MR2966711} that $L(\bar t, \bar \lambda, T)^{(-1)^{n+1}}$ has a unique reciprocal $p$-adic unit root $\pi_0(\bar t, \bar \lambda)$, which is a 1-unit. Let $\kappa \in \bb {Z}_p$ be a $p$-adic integer.  For each $\bar t$, the unit root $L$-function is defined by
\[
L_{unit}(\kappa, \bar t, T) := \prod_{\bar \lambda \in | \bb G_m^{s} / \bb F_{q_{\bar t}}|} \frac{1}{1 - \pi_0(\bar t, \bar \lambda)^\kappa T^{d_{\bar t}(\bar \lambda)}},
\]
where $\kappa$ takes values in the $p$-adic integers $\bb Z_p$. Wan's theorem tells us that this $L$-function is $p$-adic meromorphic and so may be written as a quotient of $p$-adic entire functions:
\[
L_{unit}(\kappa, \bar t, T)^{(-1)^{s+1}} = \frac{\prod_{i=1}^\infty (1 - \alpha_i(\kappa, \bar t) T)}{\prod_{j=1}^\infty (1 - \beta_j(\kappa, \bar t) T)}, \qquad \alpha_i \rightarrow 0, \beta_j \rightarrow 0 \text{ as } i, j \rightarrow \infty.
\]
Very little is known about the zeros and poles of unit root $L$-functions. In Theorem \ref{T:  Main Thm} below, we show that for each $\bar t$ and $\kappa$, $L_{unit}(\kappa, \bar t, T)^{(-1)^{s+1}}$ itself has a unique unit zero (and no unit poles), which is a 1-unit. We then study the variation of this unit root as a function of $\bar t$ and $\kappa$. We note that the variation of the unit root $L$-function with respect to the parameter $\kappa$ has been studied before in Wan's proof of Dwork's conjecture, and is connected to the Gouv\^ea-Mazur conjecture \cite{MR1122070}. On the other hand, as far as we know, the study of the $p$-adic analytic variation of the unit root $L$-function with respect to $\bar t$ is new. To state the main result, first denote by $\pi \in \overline{\bb Q}_p$ an element satisfying $\pi^{p-1} = -p$. Next, writing $G(t, \lambda, x) = f(t, x) + P(\lambda, x) = \sum t_u x^u + \sum A(\gamma, v) \lambda^\gamma x^v \in \bb F_q[x_1^\pm, \ldots, x_n^\pm, \lambda_1^\pm, \ldots, \lambda_s^\pm, \{ t_u \}_{u \in Supp(f)}]$, let  $\hat A(\gamma, v)$ be the Teichm\"uller lift of $A(\gamma, v)$ in $\bb Q_q$ for each $(\gamma, v) \in Supp(P)$. We now replace every coefficient of $A(\gamma, v)$ or $P(\lambda, x)$ with a new variable $\Lambda$: set $\c P(\Lambda, \lambda. x) := \sum_{(\gamma, v) \in Supp(P)} \Lambda_{\gamma, v} \lambda^\gamma x^v$ and
\[
H(t, \Lambda, \lambda, x) := f(t, x)  +  \c P(\Lambda, \lambda, x).
\]
Note that the series 
\[
\exp \pi H(t, \Lambda, \lambda, x) ) = \sum_{\gamma \in \bb Z^s, u \in \bb Z^n} K_{\gamma, u}(t, \Lambda) \lambda^{\gamma} x^{u}
\]
is well-defined, and its coefficients $K_{\gamma, u}(t, \Lambda)$ are themselves elements in the power-series ring  $\bb Z_p[\zeta_p][[\{ t_u \}_{u \in \c A}, \{ \Lambda_{\gamma, v} \}_{(\gamma, v) \in Supp(P)}]]$, and so converge in the open polydisk $D(0,1^-)^{|\c A| + |Supp(P)|}$ defined by the inequalities $|t_u|_p < 1$ for all $u \in \c A$ and $| \Lambda_{\gamma, v} | < 1$ for all $(\gamma, v) \in Supp(P)$. Of particular interest is $K_{0,0}(t)$, a principal $p$-adic unit for all $t$ and $\Lambda$ in the polydisk. Define $\c F(t, \Lambda) := K_{0,0}(t, \Lambda) / K_{0,0}(t^p, \Lambda^p)$ and set $\c F_m(t, \Lambda) := \prod_{i=0}^{m-1}  \c F(t^{p^i}, \Lambda^{p^i})$.

\begin{theorem}\label{T: Main Thm}
Let $\hat t$ be the Teichm\"uller lift of $\bar t$. The function $\c F(t, \Lambda)$ analytically continues to the closed polydisc $D(0,1^+)^{|\c A| + |Supp(P)|}$ defined by  $|t_u|_p \leq 1$, $u \in \c A$ and $|\Lambda_{\gamma, v}| < 1$, $(\gamma, v) \in Supp(P)$. Furthermore, $\c F_{a d(\bar t)}(\hat t, \hat A)^\kappa = \prod_{i=0}^{a d(\bar t)} \c F(\hat t^{p^i}, \hat A^{p^i})^\kappa$ is the unique unit root of $L_{unit}(\kappa, \bar t, T) ^{(-1)^{s+1}}$ at each fiber $\bar t$ and $\kappa \in \bb Z_p$, where $\c F_{a d(\bar t)}(\hat t, \hat A)$ means setting each $t_u = \hat t_u$ and $\Lambda_{\gamma, v} = \hat A(\gamma, v)$.
\end{theorem}

\medskip\noindent {\bf Remark.} It is worthwhile to compare this result to the result in \cite{MR2966711}. To that end, consider the (total) family $H(t, \Lambda, \lambda, x)$ above. For each $\bar t \in (\overline{\bb F}_q^\times)^{|\c A|}$ and $m \geq 1$, define the exponential sum
\[
S_m(H, \bar t) := \sum_{(\bar \lambda, \bar x) \in (\bb F_{q^{m \cdot deg(\bar t)}}^\times)^s \times (\bb F_{q^{m \cdot deg(\bar t)}}^\times)^n} \Psi\circ Tr_{\bb F_{q^{m \cdot deg(\bar t)}} / \bb F_q}(H(\bar t, A, \bar \lambda, \bar x)).
\]
Define by $L(H, \bar t, T) := \exp( \sum_{m \geq 1} S_m(f, \bar \lambda) \frac{T^m}{m})$ the associated $L$-function, a rational function over $\bb Q(\zeta_p)$. By \cite{MR2966711}, $L(H, \bar t, T)^{(-1)^{s+n+1}}$ has a unique $p$-adic unit root given by $\c F_{a d(\bar t)}(\hat t, \hat A)$. Conjecturally this type of relation should hold in greater generality.

\medskip\noindent {\bf Remark.} The existence of a unique $p$-adic unit root is a general result for unit root $L$-functions defined over the torus $\bb G_m^s$. This includes the classical case of $L$-functions over of exponential sums defined over the torus; see \cite[Section 3]{MR3249829} for details. 

To state this result, we use the language of $\sigma$-modules. See \cite{MR3249829} reference to the following notation.  Let $K$ be a finite extension field of $\bb Q_p$ with uniformizer $\pi$, ring of integers $R$, and residue field $\bb F_q$. Let $(M, \phi)$ be a $c \cdot \log$-convergent, nuclear $\sigma$-module over $R$, ordinary at slope zero of rank one ($h_0 = 1$) with basis $\{ e_i \}_{i \geq 0}$. Assume further the \emph{normalization condition} $\phi e_0 \equiv e_0 \text{ mod}(\pi)$ and $\phi e_i \equiv 0 \text{ mod}(\pi) \text{ for all } i \geq 1$. With this setup, it follows that the associated unit root $L$-function $L_\text{unit}(\kappa, \phi, T)^{(-1)^{s+1}}$ has a unique $p$-adic unit root (and no unit poles). To see this, we first note that by \cite[Lemma 2.1]{MR3249829} and \cite[equation (9)]{MR3249829}, $L_\text{unit}(\kappa, \phi, T)^{(-1)^{s+1}} \equiv det(1 - F_{B^{[\kappa]}} T)$ mod $\pi$. Next, it follows from the normalization condition that the matrix $B^{[\kappa]}$ takes the form $\begin{psmallmatrix}1 & 0\\0 &0 \end{psmallmatrix}$ mod $\pi$, and thus $det(1 - F_{B^{[\kappa]}} T) \equiv 1 - T$ mod $\pi$. Hence, the Fredholm determinant $det(1 - F_{B^{[\kappa]}} T)$ has a unique $p$-adic unit root proving the results.

\section{Lower deformation family}\label{S: unit root formula}

Let $f \in \bb F_q[\{ t_u\}_{u \in Supp(f)}, x_1^\pm, \ldots, x_n^\pm]$ be of the form $f(t, x) = \sum t_u x^u$. In particular, the coefficient of every monomial $x^u$ in $f$ is a new variable $t_u$. Denote by  $\Delta_\infty(f)$ the Newton polytope at infinity of $f$, defined as the convex closure of $Supp(f) \cup \{0\}$ in $\bb R^n$. Let $Cone(f)$ be the union of all rays emanating from the origin and passing through $\Delta_\infty(f)$, and set $M := M(f) := Cone(f) \cap \bb Z^n$. We define a weight function $w$ on $M$ as follows. For $u \in M$, let $w(u)$ be the smallest non-negative rational number such that $u \in w(u) \Delta(f)$. It is convenient to assume $w(u) = 1$ for all $u$ in the $x$-support of $f$. In particular this implies that  $f$ has no constant term. Let $D$ denote the smallest positive integer such that $w(M) \subset (1/D)\bb Z_{\geq 0}$. The weight function $w$ satisfies the following norm-like properties:
\begin{enumerate}
\item $w(u) = 0 $ if and only if $u = 0$.
\item $w(c u) = c w(u)$ for every $c \geq 0$.
\item $w(u + v) \leq w(u) + w(v)$ for every $u, v \in M$, with equality holding if and only if $u$ and $v$ are cofacial.
\end{enumerate}
 It is also convenient to assume the lower-order deformation  $P \in \bb F_q[\lambda_1^\pm, \ldots, \lambda_s^\pm, x_1^\pm, \ldots, x_n^\pm]$ has no constant term so the origin in $\bb {R}^n$ is not in the $x$-support of $P$. In fact, if we write  $P(\lambda, x) = \sum_{u \in M} P_u(\lambda) x^u$, then $0 <w(u) < 1$. Our lower deformation family then is defined by $G(t, \lambda, x) := f(t, x) + P(\lambda, x)$.  Set
\begin{equation}\label{E: RelPolyDef}
U := \left\{ \left( \frac{1}{1 - w(u)} \right) \gamma \in \bb Q^s \mid  (\gamma, u) \in Supp(P) \right\},
\end{equation}
and let $\Gamma := \Delta_\infty(U) \subset \bb R^s$. In a similar way to the above, define $M(\Gamma) := Cone(\Gamma) \cap \bb Z^s$ with associated polyhedral weight function $w_\Gamma$. Observe that for $\delta = \left( \frac{1}{1 - w(u)} \right) \gamma \in U$ that $w_\Gamma(\delta) < 1$.  We call $\Gamma$ the \emph{relative polytope} of the family $G(x, t)$.

\bigskip\noindent{\bf Rings of $p$-adic analytic functions.}  Let $\zeta_p$ be a primitive $p$-th root of unity. Let $\bb Q_q$ be the unramified extension of $\bb Q_p$ of degree $a := [\bb F_q : \bb F_p]$, and denote by $\bb Z_q$ its ring of integers. Then $\bb Z_q[\zeta_p]$ and $\bb Z_p[\zeta_p]$ are the ring of integers of $\bb Q_q(\zeta_p)$ and $\bb Q_p(\zeta_p)$, respectively. Let $\pi \in \overline{\bb Q}_p$  satisfy $\pi^{p-1} = -p$, and let $\tilde \pi$ be an element which satisfies $ord_p(\tilde \pi) = (p-1) / p^2$. We may have occasion to work over a purely ramified extension $\Omega_0 = \bb Q_p(\hat \pi)$  of $\bb Q_p$ with uniformizer $\hat \pi$ which contains $\bb Q_p(\zeta_p, \tilde \pi)$ and for which $\tilde \pi$ is an integral power of $\hat \pi.$ Let $\Omega = \bb Q_q(\hat \pi).$ Denote by $R$ the ring of integers  of $\Omega$, and $R_0$ the ring of integers of $\Omega_0$. Set
\[
\c O_0 := \left\{ \sum_{\gamma \in M(\Gamma)} C(\gamma) \tilde \pi^{w_\Gamma(\gamma)} \lambda^\gamma  \mid C(\gamma) \in R, C(\gamma) \rightarrow 0 \text{ as } \gamma \rightarrow \infty \right\}.
\]
(We note that the fractional powers of $\tilde \pi$ are to be understood as integral powers of a uniformizer of $R$.) Then $\c O_0$ is a ring with a discrete valuation given by
\[
\left| \sum_{\gamma \in M(\Gamma)} C(\gamma) \lambda^\gamma \tilde \pi^{w_\Gamma(\gamma)} \right| := \sup_{\gamma \in M(\Gamma)} | C(\gamma) |.
\]
Define
\[
\c C_0(\c O_0) := \left\{\xi = \sum_{\mu \in M(\bar f)} \xi(\mu) \tilde \pi^{w(\mu)} x^\mu \mid \xi(\mu) \in \c O_0, \xi(\mu) \rightarrow 0 \text{ as } \mu \rightarrow \infty \right\},
\]
an $\c O_0$-algebra. 

In the following, $q = p^a$ is an arbitrary power of $p$ (including the case when $a = 0$), so we can handle the cases of $t^q$, $t^p$, and $t$, at the same time. Define
\begin{equation}\label{E: 32}
\c O_{0, q} := \left\{ \sum_{\gamma \in M(\Gamma)} C(\gamma) \lambda^{\gamma} \tilde \pi^{w_{q\Gamma}(\gamma)} \mid C(\gamma) \in R, C(\gamma) \rightarrow 0 \text{ as } \gamma \rightarrow \infty \right\}.
\end{equation}
This ring is the same as $\c O_0$ except using a weight function defined by the dilation $q \Gamma$ (that is, $w_{q \Gamma}(\gamma) = w_\Gamma(\gamma) / q$). We note that here $\c O_{0, 1} = \c O_0$. A discrete valuation may be defined as follows. If $\xi = \sum_{\gamma \in M(\Gamma)} C(\gamma) \tilde \pi^{w_{q\Gamma}(\gamma)} \lambda^\gamma \in \c O_{0,q}$ then the valuation on $\c O_{0,q}$ is given by
\[
| \xi | := \sup_{\gamma \in M(\Gamma)} |C(\gamma)|.
\]
We may also define the space
\begin{equation}\label{E: 33}
\c C_0(\c O_{0,q}) := \left\{ \sum_{u \in M(f)} \xi_u x^u \tilde \pi^{w(u)} \mid \xi_u \in \c O_{0, q},  \xi_u \rightarrow 0 \text{ as } u \rightarrow \infty \right\}.
\end{equation}
For $\eta = \sum_{u \in M(\bar f)} \xi_u \tilde \pi^{w(u)} x^u \in \c C_0 (\c O_{0,q})$, we set
\[
|\eta| = \sup_{u \in M(f)} |\xi_u|.
\]

\bigskip\noindent{\bf Frobenius.} At present, we fix $\bar t \in (\overline{\bb F}_q)^{|A|}$, returning to variation in $\bar t$ in the last section.  Recall the notation $d({\bar t})= [\bb {F}_q({\bar t}):\bb{F}_q]$, and $q_{\bar t} = q^{d({\bar t})} $. Now let $\bar \lambda \in (\overline{\bb F}_q)^s$.  Recall we denote by $deg(\bar t)$ or $d(\bar t)$ the degree $[\bb F_q(\bar t) : \bb F_q]$. Similarly, denote by $deg(\bar \lambda)$ or $d(\bar {\lambda})$ the degree $[\bb F_q(\bar \lambda,\bar t): \bb F_q(\bar t)]$, and $q_{\bar t, \bar {\lambda}} = q^{d(\bar t)d(\bar {\lambda})}$ .

Dwork defines a splitting function by $\theta(T) := \exp \pi (T - T^p) = \sum_{i = 0}^\infty \theta_i T^i$. It is well-known that $ord_p(\theta_i) \geq \frac{(p-1)}{p^2} i$ for all $i \geq 0$. Writing
\begin{align*}
G(\bar t, \lambda, x) &= f(\bar t, x) + P(\lambda, x) \\
&= \sum \bar t_u x^u + \sum \bar A(\gamma, v) \lambda^\gamma x^v \in \bb F_{q_{\bar t}}[x_1^\pm, \ldots, x_n^\pm, \lambda_1^\pm, \ldots, \lambda_s^\pm],
\end{align*}
we let 
\[
\hat G(\hat t, \lambda, x) := \sum \hat t_u x^u + \sum \hat {A}(\gamma, v) \lambda^\gamma x^v \in R[x_1^\pm, \ldots, x_n^\pm, \lambda_1^\pm, \ldots, \lambda_s^\pm]
\]
be the lifting of $G$ by lifting the coefficients $\bar A(\gamma, u)$ and $\bar t$ by Teichm\"uller units. Set 
\begin{equation}\label{E: 38a}
 F(\hat t, \lambda, x) := \prod_{u \in Supp(f)} \theta(\hat t_u x^u) \cdot \prod_{(\gamma, v) \in Supp(P)} \theta( \hat {A}(\gamma, v) \lambda^\gamma x^v)
\end{equation}
and for any $m \geq 1$,
\begin{equation}\label{E: 38b}
 F_m(\hat t, \lambda, x) := \prod_{i=0}^{m-1}  F^{\sigma^i}(\hat t, \lambda^{p^i}, x^{p^i}),
\end{equation}
where $\sigma$ is the extension of the usual Frobenius generator of $Gal(\bb Q_q / \bb Q_p)$ to $\Omega$ with $\sigma(\hat\pi)=\hat \pi.$ Then, $ \sigma$ acts on series with coefficients in $\Omega$ by acting on these coefficients. Note that if we set $F_m(\hat t, \lambda, x) = \sum_{u \in M(f)} \mathcal{B}^m(u)x^u = \sum_{\gamma \in M_(\Gamma), u \in M(f)} \mathcal{B}^m (\gamma, u) \lambda^{\gamma} x^u ,$ then 
\[
ord_p (\mathcal{B}^m (\gamma, u)) \geq \frac{w_{\Gamma}(\gamma) + w(u)}{p^{m-1}} \cdot \frac{p-1}{p^2}.
\]

Define $\psi_x$ by $\sum C(u) x^u \mapsto \sum C(p u) x^u$. Set
\[
\alpha_1 :=  \sigma^{-1} \circ \psi_x \circ F(\hat t, \lambda, x)
\]
A similar argument to that in \cite{MR3239170} demonstrates that $\alpha_1$ maps $\sigma^{-1}$-semilinearly $\c C_0(\c O_0)$ into $\c C_0( \c O_{0, p})$. Similarly, for $m \geq 1$, if we define
\[
\alpha_m := \sigma^{-m} \circ \psi_x^m \circ F_m(\hat t, \lambda, x),
\]
then $\alpha_m$ maps $\c C_0(\c O_0)$ into $\c C_0(\c O_{0, p^m})$. In particular, $\alpha_m(\tilde{\pi}^{w(v)} x^v) = \sum_{u \in M(f)} \tilde{\pi}^{w(v)-w(u)}\mathcal{B}^m(p^m u -v) \tilde{\pi}^{w(u)} x^u$, with $ord_p(\tilde{\pi}^{w(v)-w(u)}\mathcal{B}^m(p^m u -v)  \geq \frac{(p^m - 1)w(u) + (p^{m-1} -1)w(v)}{p^{m-1}} ord_p( \tilde{\pi})$. Summarizing, we have in $\c C_0(\c O_{0,p^m}),  |\alpha_m(\tilde{\pi}^{w(v)}x^v)| \leq |\tilde{\pi}|^{w(v) \frac{p^{m-1} -1}{p^{m-1}}}$.

\bigskip\noindent{\bf Fibers.} Define
\[
\alpha_{\bar t, \bar \lambda} := \psi_x^{a d(\bar t) d(\bar \lambda)} \circ F_{a d(\bar t) d(\bar \lambda)}(\hat t, \hat \lambda, x),
\]
where $\hat t$ and $\hat \lambda$ are the Teichm\"uller representatives of $\bar t$ and $\bar \lambda$, respectively. Notice that $\alpha_{\bar t, \bar \lambda}$ is an endomorphism of $\c C_0(\hat \lambda)$, where $\c C_0(\hat \lambda)$ denotes the space obtained from $\c C_0(\c O_0)$ by applying the map on  $\c O_0$  which sends  $\lambda \ \textup{to} \ \hat \lambda.$

To relate the $L$-function $L(\bar t, \bar \lambda, T)$ to the operator $\alpha_{\bar t, \bar \lambda}$ it is convenient to introduce the following operation: for any function $g(T)$, define $g(T)^{\delta_q} := g(T) / g(qT)$. Set $q_{\bar t, \bar \lambda} := q^{d(\bar t) d(\bar \lambda)}$.  Dwork's trace formula states
\[
(q_{\bar t, \bar \lambda}^m-1)^n Tr(\alpha_{\bar t, \bar \lambda}^m \mid \c C_0(\hat \lambda)) = \sum_{\bar x \in \left(\bb F_{q_{\bar t, \bar \lambda}^m}^*\right)^n} \Psi \circ Tr_{\bb F_{q_{\bar t, \bar \lambda}^m} / \bb F_q}( G(\bar t, \bar \lambda, \bar x))
\]
Equivalently,
\[
L(\bar t, \bar \lambda, T)^{(-1)^{n+1}} = det(1 - \alpha_{\bar t, \bar \lambda} T \mid \c C_0(\hat \lambda))^{\delta_{q_{\bar t, \bar \lambda}}^n}.
\]
This is a rational function, and it is well-known that $L(\bar t, \bar \lambda, T)^{(-1)^{n+1}}$ has a unique unit (reciprocal) root $\pi_0(\bar t, \bar \lambda)$ (see \cite{MR2966711}  for example). This unit root is a $1$-unit, so it makes sense to define, for any $p$-adic integer $\kappa$, the unit root $L$-function at the fibre $\bar t$:
\[
L_{\text{unit}}(\kappa, \bar t, T) := \prod_{\bar \lambda \in |\bb G_m^{s} / \bb F_q(\bar t)|} \frac{1}{1 - \pi_0(\bar t, \bar \lambda)^\kappa \> T^{deg(\bar \lambda)}}.
\]
Denote the roots of $det(1 - \alpha_{\bar t, \bar \lambda} T \mid \c C_0(\hat \lambda) )$ by $\pi_i(\bar t, \bar \lambda)$, and order them such that $ord_p \> \pi_i(\bar t, \bar \lambda) \leq ord_p \> \pi_{i+1}(\bar t, \bar \lambda)$ for $i \geq 0$. For each $m \geq 0$, define
\[
L^{(m)}(\kappa, \bar t, T) := \prod_{\bar \lambda \in | \bb G_m^s / \bb F_{q_{\bar t}}|} \prod (1 - \pi_0(\bar t, \bar \lambda)^{\kappa - r - m} \pi_{i_1}(\bar t, \bar \lambda) \cdots \pi_{i_r}(\bar t, \bar \lambda) \cdot \pi_{j_1}(\bar t, \bar \lambda) \cdots \pi_{j_m}(\bar t, \bar \lambda) T^{deg(\bar \lambda)})^{-1}
\]
where the inner product runs over all $r \geq 0$, $1 \leq i_1 \leq i_2 \leq \cdots$, and $0 \leq j_1 < \cdots < j_m$. Note that the factors indexed by the various $i_k$ are allowed to repeat, whereas the factors with indices $j_l$ are distinct. Intuitively, the inner product is $det(1 - Sym^{\kappa - m} \alpha_{\bar t, \bar \lambda} \otimes \wedge^m \alpha_{\bar t, \bar \lambda} T)$. From \cite[Lemma 2.1]{MR3249829},
\begin{equation}\label{E: Adams}
L_\text{unit}( \kappa, \bar t, T) = \prod_{i = 0}^\infty L^{(i)}(\kappa, \bar t, T)^{(-1)^{i-1} (i-1)}
=L^{(0)}(\kappa,\bar t, T) \prod_{i \geq 2}L^{(i)}(\kappa, \bar t, T)^{(-1)^{i-1}(i-1)}.
\end{equation}
In the next section, we will show each $L^{(i)}$ with $i \geq 1$ has no unit root or pole, whereas $L^{(0)}$ will. This will show $L_\text{unit}(\kappa, \bar t, T)^{(-1)^{s+1}}$ has a unique unit root.

\section{Infinite symmetric powers}\label{S: inf sum pow}

Denote by $\c S(\hat \lambda) := R[\hat \lambda][[ \{e_u\}_{u \in M \setminus \{0\}}]]$ the formal power series ring over $R[\hat \lambda]$ in the variables $\{e_u \}_{u \in M \setminus \{0\}}$ which are formal symbols indexed by the $M \setminus \{0\}$.  We equip this ring  with the sup-norm on coefficients (in $R[\hat \lambda]$). This ring will play the role of the formal infinite symmetric power of $\c C_0(\hat \lambda) $ over $R[\hat \lambda]$ in a way we describe below. It is convenient to write the monomials of degree $r$ in the variables $\{ e_u \}$ using the notation $e_{\b u} := e_{u_1} \cdots e_{u_r}$, where $u_1, \ldots, u_r \in M(f) \setminus \{0\}$ for  $r \geq 0$. It helps to fix ideas to assume we have a linear order on $M(f) \setminus \{0\}$ with the property that if $w(u) \leq w(v)$ for $u, v \in M(f) \setminus \{0\}$, then $u \leq v$. We may extend this to all of $M(f)$ by taking $0$ as the least element. We emphasize then in the notation $e_{\b u} := e_{u_1} \cdots e_{u_r}$ for a monomial of degree $r$ we have $0 < u_1 \leq u_2 \leq \cdots \leq u_r$, and we allow the variables to repeat. When $r = 0$ we understand there is only the monomial $1$ of degree 0. We extend the weight function $w$ to such monomials by defining, for $e_{\b u} := e_{u_1} \cdots e_{u_r}$, the weight $w(\b u) := w(u_1) + \cdots + w(u_r)$. Denote by $\c S(M)$ the set of all indices $\b u$ corresponding to monomials $e_{\b u}$. We emphasize that we will often equate elements $\b u \in S(M)$ with the monomials $e_{\b u}$; it should be clear from the context which meaning is desired. We may assume $\c S(M)$ has a linear order defined on it such that the weight $w(\b u)$ is non-decreasing and such that the restriction of this linear order to $M(f)$ is our earlier linear order. 

We may identify $\c C_0(\hat \lambda)$ as an $R[\hat \lambda]$-submodule  of $\c S(\hat \lambda)$ by defining an $R[\hat \lambda]$-linear map 
\[
\Upsilon: \c C_0(\hat \lambda) \rightarrow \c S(\hat \lambda) \qquad  \text{via} \qquad  \sum_{u \in M(f)} \xi_u \tilde \pi^{w(u)} x^u \longmapsto \xi_0 + \sum_{u \in M(f)\setminus \{0\}} \xi_u e_u.
\]
That is, the image $\Upsilon(\c C_0(\hat \lambda))$ consists of the powers series with support in the monomials of $\c S(\hat \lambda)$ of degree $\leq 1$ and with coefficients $\{\xi_u \}_{u \in M(f)} \subset R[\hat \lambda]$ satisfying $\xi_u \rightarrow 0$ as $u \rightarrow \infty$. Note that $\Upsilon(\tilde \pi^{w(u)}x^u) = e_u$ for $u \in M \setminus \{ 0 \}$, and $\Upsilon(1) := 1$. 
Define the $R[\hat \lambda]$-subalgebra of $\c S(\hat \lambda) $
\[
\c S_0(\hat \lambda) := \left\{ \xi = \sum_{\b u \in \c S(M)} \xi(\b u) e_{\b u} \mid \xi(\b u) \in R[\hat \lambda], \xi(\b u) \rightarrow 0 \text{ as } w(\b u) \rightarrow \infty \right\}.
\]
Hence, $\Upsilon(\c C_0(\hat \lambda)) \subset \c S_0(\hat \lambda)$. Note that we may write  $\alpha_{\bar t, \bar \lambda}(1) = 1 + \eta(x)$ for some element $\eta \in \c C_0(\hat \lambda)$ satisfying $|\eta| < 1$ and with support of $\eta$ in $M(f)\setminus \{ 0 \}$. For $\xi = \sum \xi(\b u) e_{\b u} \in \c S_0(\hat \lambda)$, define $| \xi | := \sum_{u \in S(M)}  | \xi(\b u) |$, which makes $\c S_0(\hat \lambda)$ a $p$-adic Banach algebra over $R[\hat \lambda]$. Then for any $\zeta \in \c C_0(\hat \lambda)$, $|\Upsilon(\zeta)| = |\zeta|$. It follows that $(\Upsilon \circ \alpha_{\bar t, \bar \lambda}( 1 ))^{\tau}$ is defined and belongs to $\c S_0(\hat \lambda)$ for any $\tau \in \bb Z_p$. Define $[\alpha_{\bar t, \bar \lambda}]_\kappa: \c S_0(\hat \lambda) \rightarrow \c S_0(\hat \lambda)$ by extending linearly over $R[\hat \lambda]$ the action on monomials of degree $r$ 
\[
[\alpha_{\bar t, \bar \lambda}]_\kappa( e_{u_1} \cdots e_{u_r} ) := (\Upsilon \circ \alpha_{\bar t, \bar \lambda} (1) )^{\kappa - r} (\Upsilon \circ \alpha_{\bar t, \bar \lambda} ( \tilde \pi^{w(u_1)} x^{u_1})) \cdots ( \Upsilon \circ \alpha_{\bar t, \bar \lambda} (\tilde \pi^{w(u_r)} x^{u_r})).
\]
By a similar argument to \cite[Corollary 2.4, part 2]{MR3249829}, 
\[
det(1 - [\alpha_{\bar t, \bar \lambda}]_\kappa T \mid \c S_0(\hat \lambda) ) = \prod_{r = 0}^\infty \prod \left(1 - \pi_0(\bar t, \bar \lambda)^{\kappa - r} \pi_{i_1}(\bar t, \bar \lambda) \cdots \pi_{i_r}(\bar t, \bar \lambda)  T \right)
\]
where the inner product runs over all multisets $\{i_1, \ldots, i_r\}$ of positive integers of cardinality $r$ satisfying $1 \leq i_1 \leq i_2 \leq \cdots$.

\bigskip\noindent{\bf Infinite symmetric power on the family.}  Denote by $\c S(\c O_0) := \c O_0[[ \{ e_u \}_{u \in M \setminus \{0\}}]]$, the formal power series ring supported by the monomials $\c S(M)$, with coefficients in the ring $\c O_0$.  As in the constant fibre case above, this ring is equipped with the sup-norm on coefficients.  Define the $p$-adic Banach algebra over $\c O_0$,
\begin{align*}
\c S_0(\c O_0) :&= \{ \xi =  \sum_{\b u \in \c S(M)} \xi(\b u) e_{\b u} \mid \xi(\b u) \in \c O_0, \xi(\b u) \rightarrow 0 \text{ as } w(\b u) \rightarrow \infty \} \\
&= \{ \xi = \sum_{\gamma \in M(\Gamma), \b u \in \c S(M)} C(\gamma, \b u) \tilde \pi^{w_{\Gamma}(\gamma)} \lambda^\gamma e_{\b u} \mid C(\gamma, \b u) \in R, C(\gamma, \b u) \rightarrow 0 \text{ as } w_\Gamma(\gamma) + w(\b u) \rightarrow \infty \},
\end{align*}
and similarly, for any $q = p^a$ an arbitrary power of $p$ (including the case when $a = 0$),
\[
\c S_0(\c O_{0, q}) := \{ \sum_{\b u \in \c S(M)} \xi(\b u) e_{\b u} \mid \xi(\b u) \in \c O_{0, q}, \xi(\b u) \rightarrow 0 \text{ as } w(\b u) \rightarrow \infty \}.
\]
Note that $\c S_0(\c O_{0, q})$ is a $p$-adic Banach algebra over $\c O_{0, q}$ with  $\c S(M)$ an orthonormal basis. 
We embed $\c C_0(\c O_{0, q}) \hookrightarrow \c S_0(\c O_{0, q})$ via a map $\Upsilon$ defined in the same way as on the fibers.  Again,  $(\Upsilon \circ \alpha_m(1))^{\tau} \in \c S_0(\c O_{0, p^m})$ for any $\tau \in \bb Z_p$. We define a map $[\alpha_m]_\kappa : \c S_0(\c O_0) \rightarrow \c S_0(\c O_{0, p^m})$ as follows. On a basis element $e_{\b u} = e_{u_1} \cdots e_{u_r}$ with $r > 0$ and $0 < u_1 \leq \cdots \leq u_r$, 
\[
[\alpha_m](e_{\b u}) := [\alpha_m]_\kappa( e_{u_1} \cdots e_{u_r} ) := (\Upsilon \circ \alpha_m( 1 ))^{\kappa - r} (\Upsilon \circ \alpha_m (\tilde \pi^{w(u_1)} x^{u_1}) ) \cdots (\Upsilon \circ \alpha_m (\tilde \pi^{w(u_r)} x^{u_r})).
\]
If $r = 0$, 
\[
[\alpha_m]_\kappa(1) := \Upsilon( \alpha_m(1))^\kappa.
\]

We may calculate an estimate for $\alpha_m( \tilde \pi^{w(u)} x^u)$, where we recall $\alpha_m := \sigma^{-m} \circ \psi_x^m \circ F_m(\bar t, \lambda, x)$. As noted earlier, we may write 
\begin{equation}\label{E: splitting}
F_m(\hat t, \lambda, x) = \sum_{\gamma \in M(\Gamma), v \in M(f)} B(\gamma, v) \tilde \pi^{( w_\Gamma(\gamma) + w(v)) / p^{m-1}} \lambda^\gamma x^v,
\end{equation}
with $ord_p \> B(\gamma, r) \geq 0$, and set $\mathcal{B}^m(\gamma,v) =  B(\gamma, v) \tilde \pi^{( w_\Gamma(\gamma) + w(v)) / p^{m-1}} $. So
\begin{align*}
\alpha_m( \tilde \pi^{w(u)} x^u) &= \psi_x^m \left( F_m(\hat t, \lambda, x) \cdot \tilde \pi^{w(u)} x^u \right) \\
&= \sum  \tilde \pi^{(w_\Gamma(\gamma) + w(p^m v - u)) / p^{m-1} + w(u) - w_\Gamma(\gamma)/p^{m-1} - w(v)} B(\gamma, p^m v - u) \cdot \tilde \pi^{w_\Gamma(\gamma) / p^{m-1}} \lambda^\gamma \cdot \tilde \pi^{w(v)} x^v.
\end{align*}
We note that
\begin{align*}
\frac{w(p^m v - u)}{p^{m-1}} + w(u) - w(v) &\geq p w(v) - \frac{w(u)}{p^{m-1}} + w(u) - w(v) \\
&\geq (p-1) w(v) + \frac{p^{m-1}-1}{p^{m-1}} w(u).
\end{align*}
Hence,
\begin{equation}\label{E: upsilon estimate}
| \Upsilon( \alpha_m( \tilde \pi^{w(u)} x^u) ) | \leq | \tilde \pi |^{ \frac{p^{m-1}-1}{p^{m-1}} w(u)}
\end{equation}

The $R$-linear map $\psi_\lambda: \c S_0(\c O_{0, p}) \rightarrow \c S_0(\c O_0)$ is defined by 
\[
\psi_\lambda: \quad \sum_{\gamma \in M(\Gamma), \b u \in \c S(M)} A(\gamma, \b u)  \lambda^\gamma e_{\b u} \longmapsto \sum_{\gamma \in M(\Gamma), \b u \in \c S(M)} A(p \gamma, \b u)  \lambda^\gamma e_{\b u}
\]
We may in the usual manner view $\c S_0(\c O_0)$ as a $p$-adic Banach space over $R$ with orthonormal basis $\{ \tilde \pi^{w_\Gamma(\gamma)} \lambda^\gamma e_{\b u} \mid \gamma \in M(\Gamma), \b u \in \c S(M) \}$. Then 
\[
\beta_{\kappa, \bar t} := \psi_\lambda^{a d(\bar t)} \circ [\alpha_{a d(\bar t)}]_\kappa: \c S_0(\c O_0) \rightarrow \c S_0(\c O_0)
\]
is a completely continuous operator (over $R$).  Set $\c B := \{ e_{\b u} \mid  \b u \in \c S(M)\}$. Let $B^{[\kappa]}_{\bar t}(\lambda)$ be the matrix of $[\alpha_{a d(\bar t)}]_\kappa$ with respect to $\c B$, the basis of $\c S_0(\c O_0)$ over $\c O_0$ (as well as $\c S_0(\c O_{0, p^m})$ over $\c O_{0, p^m}$). The entries of $B_{\bar t}^{[\kappa]}(\lambda)$ are series with support in $\c B$ and coefficients in $\c O_{0, p^m}$ (which tend to 0 as $w(\b u) \rightarrow \infty$). We may write $B^{[\kappa]}_{\bar t}(\lambda) = \sum_{\gamma \in M(\Gamma)} b_\gamma^{[\kappa]} \lambda^\gamma$, where $b_\gamma^{[\kappa]}$ is a matrix with rows and columns indexed by $M(\Gamma)$ and entries in $R$. We define the matrix $F_{B^{[\kappa]}_{\bar t}} := (b_{q_{\bar t} \gamma - \mu}^{[\kappa]} )_{(\gamma, \mu)}$ indexed by $\gamma, \mu \in M(\Gamma)$, and we set $b_{q_{\bar t} \gamma - \mu}^{[\kappa]} := 0$ if $q_{\bar t} \gamma - \mu \not\in M(\Gamma)$. Note that $F_{B^{[\kappa]}_{\bar t}}$ is a matrix with entries in $R$ whose $(\gamma, \mu)$ entry is again a matrix in $R$ with rows and columns indexed by $M(\Gamma)$. As we showed in \cite[\S 2.3]{MR3239170}, $F_{B^{[\kappa]}_{\bar t}}$ is the matrix of the completely continuous operator $\beta_{\kappa, \bar t}$, and as such it has a well-defined Fredholm determinant. In particular, the Dwork trace formula gives
\begin{align*}
(q_{\bar t}^m - 1)^s Tr( \beta_{\kappa, \bar t}^m ) &= (q_{\bar t}^m - 1)^s Tr( F_{B^{[\kappa]}_{\bar t}}^m) \\
&= \sum_{\hat \lambda^{q_{\bar t}^m} = \hat \lambda} Tr \left( B^{[\kappa]}_{\bar t}(\hat \lambda^{q_{\bar t}^{m-1}}) \cdots B^{[\kappa]}_{\bar t}(\hat \lambda^{q_{\bar t}}) B^{[\kappa]}_{\bar t}(\hat \lambda) \right) \\
&= \sum_{\substack{ \bar \lambda \in (\bb F_{q_{\bar t}^m}^*)^s \\ \hat \lambda = \text{Teich}(\bar \lambda)}}  Tr \left( [ \alpha_{\bar t, \bar \lambda} ]_{\kappa}^m \mid \c S_0(\hat \lambda) \right).
\end{align*}
Using an argument similar to that succeeding \cite[Equation 8]{MR3249829}, it follows that
\begin{equation}\label{E: A}
L^{(0)}(\kappa, \bar t, T)^{(-1)^{s+1}} = det(1 - \beta_{\kappa, \bar t} T)^{\delta_{q_{\bar t}}^{s}}.
\end{equation}
Since the Fredholm determinant $det(1 - \beta_{\kappa, \bar t} T)$ is $p$-adically entire, this demonstrates the meromorphic continuation of $L^{(0)}(\kappa, \bar t, T)$.  Since the matrix of $\beta_{\kappa, \bar t}$ shows that $det(1 - \beta_{\kappa, \bar t} T)$ has a unique unit root,  it follows  that $L^{(0)}(\kappa, \bar t, T)^{(-1)^{s+1}}$ has a unique unit root equal in fact to the unique unit root of $det(1 - \beta_{\kappa, \bar t} T).$

 In a similar way, define on the space $\c S_0(\c O_0) \otimes \wedge^m \c C_0(\c O_0)$, the operator $\beta_{\kappa, \bar t}^{(m)} := \psi_\lambda^{a d(\bar t)} \circ ([\alpha_{a d(\bar t})]_{\kappa-m} \otimes \wedge^m \alpha_{a d(\bar t)})$. Then
\[
L^{(m)}(\kappa, \bar t, T)^{(-1)^{s+1}} = det(1 - \beta_{\kappa, \bar t}^{(m)} T)^{\delta_{q_{\bar t}}^{s}}.
\]
In particular, for $m \geq 2$, due to the wedge product,  $L^{(m)}(\kappa, \bar t, T)^{(-1)^{s+1}}$ has no zeros or poles on the closed unit disk. Hence, by (\ref{E: Adams}), we have:

\begin{theorem}\label{T: unique unit}
$L_\text{unit}(\kappa, \bar t, T)^{(-1)^{s + 1}}$ has a unique $p$-adic unit root which in fact is the unique unit root of $L^{(0)}(\kappa, \bar t, T)^{(-1)^{s+1}}$.
\end{theorem}

\section{Dual theory} 

In this section, we define a dual theory for the operator $\beta_{\kappa, \bar t}$ acting on $\c S_0(\c O_0)$. We begin by defining a dual map to $\alpha_{a d(\bar t)}$. For $q = p^a$ an arbitrary power of $p$ (including the case $a = 0$) define the $\c O_{0, q}$-module
\[
\c C_0^*(\c O_{0, q}) := \left\{ \sum_{u \in M(f)} \xi(u) \tilde \pi^{-w(u)} x^{-u} \mid \xi(u) \in {\c O}_{0, q}  \right\},
\]
equipped with the sup-norm on the set of coefficients $\{ \xi(u) \}_{u \in M(f)}$.  Define the projection (or truncation) map 
\[
\text{pr}_{M(f)}: \quad \sum_{u \in \bb Z^n} A(u) x^{-u} \longmapsto \sum_{u \in M(f)} A(u) x^{-u}.
\]
For each $m \geq 1$, define
\[
\alpha_m^* := \text{pr}_{M(f)} \circ F_m(\hat t, \lambda, x) \circ \Phi_x^m \circ \sigma^m,
\]
where $\sigma \in Gal(\Omega/\Omega_0)$ acts on coefficients (as mentioned above), and $\Phi_x$ acts on monomials by $\Phi_x(x^u) := x^{p u}$.

\begin{lemma}\label{L: better conv}
$\alpha_m^*: \c C_0^*(\c O_{0, p^m}) \rightarrow \c C_0^*({\c O}_{0, p^m})$ is a linear map over $\c O_{0, p^m}$. Furthermore, writing 
\[
\alpha_m^*(\tilde \pi^{-w(v)} x^{-v}) = \sum_{z \in M(f)} C_v(z) \tilde \pi^{-w(z)} x^{-z}
\]
with $C_v(z) \in \c O_{0, p^m}$, then $C_v(z) \rightarrow 0$ in $\c O_{0, p^m}$ as $w(v) \rightarrow \infty$. In addition, we may write $\alpha_m^*(1) = 1 + \eta_m^*(\lambda, x),$  with $\eta_m^*(\lambda, x) \in \c C^*_0(\c O_{0,p^m})$ having $|\eta_m^*| \leq |\tilde{\pi}|.$
\end{lemma}

\begin{proof}
We consider $\alpha_m^*(\tilde \pi^{-w(v)} x^{-v})$ with $v \in M(f)$. Using (\ref{E: splitting}), we may write this as
\[
\alpha_m^*(\tilde \pi^{-w(v)} x^{-v}) = \sum_{z \in M(f), \gamma \in M(\Gamma)} B(\gamma, -z + p^m v) \tilde \pi^{w(\gamma) / p^{m-1}} \lambda^\gamma \cdot \tilde \pi^{w(z) + ( w(-z + p^m v) / p^{m-1})} \cdot \tilde \pi^{-w(z)} x^{-z}.
\]
Since 
\[
-w(v) + w(z) + \frac{1}{p^{m-1}} w(-z + p^m v) \geq \frac{p^{m-1} - 1}{p^{m-1}} w(z) + p w(v),
\]
we see that 
\begin{equation}\label{E: an estimate}
\alpha_m^*(\tilde \pi^{-w(v)} x^{-v}) = \tilde \pi^{(p-1) w(v)} \zeta_v^*(\lambda),
\end{equation}
where $\zeta_v^*(\lambda,x) \in \c C_0^*(\c O_{0, p^m})$. 

If $\xi^* \in \c C_0^*(\c O_{0, p^m})$ with $\xi^* = \sum_{v \in M(f)} A_v(\lambda) \tilde \pi^{-w(v)} x^{-v}$, then 
\[
\alpha_m^*( \xi^*) = \sum_{v \in M(f)} \tilde \pi^{(p-1)w(v)} A_v(\lambda) \eta_v^*(\lambda) \in \c C_0^*(\c O_{0, p^m}).
\]
Finally, note that by the above,
\[
\alpha_m^*(1) = 1 + \sum_{\gamma \in M(\Gamma)-0} B(\gamma,0) \tilde{\pi} ^{w(\gamma)/p^{m-1}} \lambda^{\gamma} + \sum_{z \in M(f)-0, \gamma \in M(\Gamma)} B(\gamma, -z) \tilde{\pi}^{w(z) + (w(-z)/p^{m-1})}(\tilde{\pi}^{w(\gamma)/p^{m-1}} \lambda^{\gamma}) (\tilde{\pi}^{w(-z)} x^{-z}).
\] 
This proves the lemma.
\end{proof}

Define
\[
\c A_0 := \left\{ \sum_{\gamma \in M(\Gamma)} A(\gamma) \lambda^\gamma : A(\gamma) \in R \text{ and } A(\gamma) \rightarrow 0 \text{ as } w(\gamma) \rightarrow \infty \right\}.
\]
For $q_1$ and $q_2$ any two powers of the prime $p$, define a pairing $(\cdot, \cdot): \c C_0(\c O_{0, q_1}) \times \c C_0^*(\c O_{0, q_2}) \rightarrow \c A_0$ by
\[
(\xi, \xi^*) := \text{ the constant term with respect to $x$ of the product $\xi \cdot \xi^*$.}
\]
This product is well-defined since if $\{\eta_1(v)\}_{v \in M(\Gamma)} \subset \c O_{0, q_1}$ with $\eta_1(v) \rightarrow 0$ as $w(v) \rightarrow \infty$, and $\{ \eta_2(v) \}_{v \in M(\Gamma)} \subset \c O_{0, q_2}$, then $\sum_{v \in M(\Gamma)} \eta_1(v) \eta_2(v) \in \c A_0$. Next, observe that for $\xi \in \c C_0(\c O_0)$ and $\xi^* \in \c C_0^*(\c O_{0, p^m})$, writing $F_m$ for $F_m(\hat t, \lambda, x)$, then
\begin{equation}\label{E: pairing1}
( (\psi_x^m \circ F_m) \xi, \xi^*) = ( F_m \xi, \Phi_x^m \xi^*) = (\xi, (\text{pr}_{M(f)} \circ F_m \circ \Phi_x^m)( \xi^*)).
\end{equation}

\bigskip\noindent{\bf Symmetric powers.} We construct in a now familiar manner formal $k$-th symmetric powers of $\c C_0(\c O_0)$ and $\c C_0^*(\c O_{0, p^m})$ over $\c O_0$. Similar to the construction used above, we consider a linear order on $\{ u \in M(f) \}$ under which the weight is nondecreasing, say $0 = u_0 \leq u_1 \leq \cdots$. We will for convenience of notation write the ``basis'' as $\{ E_u := \tilde \pi^{w(u)} x^u \mid u \in M(f) \}$, and the $k$-th symmetric power of the basis as
\[
E_{\b u} := E_{u_{j_1}} E_{u_{j_2}} \cdots E_{u_{j_k}}, \qquad 0 \leq j_1 \leq j_2 \leq \cdots \leq j_k,
\] 
where $\b u$ runs over multisets of indices of cardinality $k$, say
\[
\{ \b u = (u_{j_1}, u_{j_2}, \ldots, u_{j_k}) \mid 0 \leq u_{j_1} \leq u_{j_2} \leq \cdots \leq u_{j_k} \}.
\]
Defining
\[
Sym^k_{\c O_0} \c C_0(\c O_0) := \left\{ \xi = \sum_{|\b u| = k} \xi_{\b u}(\lambda) E_{\b u} \mid \xi_{\b u}(\lambda) \in \c O_0, \xi_{\b u}(\lambda) \rightarrow 0 \text{ as } w(\b u) \rightarrow +\infty  \right\},
\]
then we define the map
\[
Sym^k \alpha_m: Sym^k_{\c O_0} \c C_0(\c O_0) \rightarrow Sym^k_{\c O_{0, p^m}} \c C_0(\c O_{0, p^m})
\]
as follows. Let
\begin{align*}
\alpha_m(\tilde \pi^{w(u)} x^u) &= \sum_{v \in M(f)} \mathcal{A}^m_{v, u}(\lambda) \tilde \pi^{w(v)} x^v \\
&= \sum_{v \in M(f)} \mathcal{A}^m_{v, u}(\lambda) E_v.
\end{align*}
We know from Section \ref{S: unit root formula} that 
\[
\mathcal{A}^m_{u,v} = \sum_{\gamma \in M(\Gamma), v \in M(f)} \tilde{\pi}^{(u)-w(v)} \mathcal{B}^m(\gamma, p^mv-u)\lambda^{\gamma}
\]
Then
\[
Sym^k \alpha_m(E_{u_{j_1}} E_{u_{j_2}} \cdots E_{u_{j_k}}) = \sum \mathcal{A}^m_{v_{l_1}, u_{j_1}}(\lambda) \cdots \mathcal{A}^m_{v_{l_k}, u_{j_k}}(\lambda) E_{v_{l_1}} \cdots E_{v_{l_k}},
\]
where the sum runs over all $v_{l_i} \in M(f)$ for each $i$, $1 \leq i \leq k$. Since by above, $|\alpha_m(\tilde \pi^{w(u)} x^u)| \leq |\tilde \pi|^{w(u) \frac{p^{m -1}-1}{p^{m-1}}}, $ therefore $Sym^k(\alpha_m)$ is a completely continuous map. The map $\Upsilon$ may be extended to $Sym^k_{\c O_0}(\c C_0(\c O_0)) \hookrightarrow \c S_0(\c O_0)$ as follows. For $\b u = (u_{j_1}, \ldots, u_{j_k})$ an ordered multiset of cardinality $k$ with elements in $M(f)$, set 
\[
\Upsilon(E_{\b u}) = 
\begin{cases}
e_{\b u} & \text{if } j_1 > 0 \\
e_{u_{j_{r+1}}} e_{u_{j_{r+2}}} \cdots e_{u_{j_k}} & \text{if } j_1 = j_2 = \cdots = j_r = 0.
\end{cases}
\]
Thus $\Upsilon(Sym^k_{\c O_0} \c C_0(\c O_0))$ consists of all power-series with coefficients in $\c O_0$ and support in monomials $e_{\b u}$ of degree $\leq k$, with coefficients going to $0$ as $w(\b u) = w(u_1) + \cdots + w(u_r) \rightarrow \infty$.

We have as well a dual variant 
\[
Sym^k_{\c O_{0, p^m}} \c C_0^*(\c O_{0, p^m}) := \{ \sum_{ |\b u| = k} A_{\b u}(\lambda) E_{\b u}^*\mid A_{\b u}(\lambda) \in \c O_{0, p^m} \}
\]
where we denote $E_u^* := \tilde \pi^{-w(u)} x^{-u}$ for each $u \in M(f)$, and using the linear order above write for each multiset $\b u = (u_{j_1}, \ldots, u_{j_k})$ of cardinality $k$ of indices, with $j_1 \leq \cdots \leq j_k$ we set $E_{\b u}^* := E_{u_1}^* \cdots E_{u_k}^*$. Then
\[
Sym^k_{\c O_{0, p^m}} \c C_0^*(\c O_{0, p^m}) = \{ \sum_{ |\b u| = k} \xi(\b u) E_{\b u}^* \mid \xi(\b u) \in \c O_{0, p^m} \},
\]
there being no requirement here that the coefficients tend to 0 as $w(\b u) \rightarrow \infty$. Since $\alpha_m^*: \c C_{0}^*(\c O_{0, p^m}) \rightarrow \c C_{0}^*(\c O_{0, p^m})$, we may define for $\b u = (u_{j_1}, \ldots, u_{j_k})$,
\[
Sym^k(\alpha_m^*)(E_{\b u}^*) = \sum \mathcal{A}^*_{v_{l_1}, u_{j_1}}(\lambda) \mathcal{A}^*_{v_{l_2}, u_{j_2}}(\lambda) \cdots \mathcal{A}^*_{v_{l_k}, u_{j_k}}(\lambda) E_{\b v}^*
\]
where $\b v = (v_{l_1}, \ldots v_{l_k})$, the sum runs over $v_{l_i} \in \{ \tilde \pi^{- w(u)} x^{-u} \mid u \in M(f) \}$, and where $\alpha_m^*(\tilde \pi^{-w(u)} x^{-u}) = \sum_{v \in M(f)} \mathcal{A}^*_{u, v}(\lambda) \tilde \pi^{-w(v)} x^{-v}$. The map $Sym^k(\alpha_m^*)$ then is defined on $Sym^k_{\c O_{0, p^m}}$ since as we noted earlier in (\ref{E: an estimate}), $|\alpha_m^*(\tilde \pi^{-w(u)} x^{-u}) | \leq |\tilde \pi|^{w(u)(p-1)}$.

We extend the pairing above to these symmetric power spaces by ``linearly'' extending the following: for decomposable elements $\xi = \xi_1 \cdots \xi_k \in Sym^k_{\c O_{0, q_1}} \c C_0(\c O_{0, q_1})$ and $\xi^* = \xi_1^* \cdots \xi_k^* \in Sym^k_{\c O_{0, q_2}} \c C_0^*(\c O_{0, q_2})$,
\begin{equation}\label{E: sym pairing def}
(\xi, \xi^*) := (\xi_1 \cdots \xi_k, \xi_1^* \cdots \xi_k^*)_k := \frac{1}{k!} \sum_{\sigma \in S_k} \prod_{i=1}^k (\xi_i, \xi_{\sigma(i)}^*),
\end{equation}
where $S_k$ denotes the symmetric group on $k$ letters. This pairing $(\cdot, \cdot)_k$ is well-defined since $\c A_0$ is a ring. Observe that it follows from (\ref{E: pairing1}) that for $\xi \in Sym^k \c C_0(\c O_0)$ and $\xi^* \in Sym^k \c C_0^*(\c O_{0, q_{\bar t}})$, 
\begin{equation}\label{E: dual sym}
(Sym^k \> \alpha_{a d(\bar t)} \xi, \xi^*)_k = (\xi, Sym^k \>\alpha_{a d(\bar t)}^* \xi^*)_k.
\end{equation}

\bigskip\noindent{\bf Infinite symmetric powers.} Denote by $\c S^*_0(\c O_0) := \c O_0[[ e_u^* : u \in M \setminus \{0\}]]$ the formal power series ring over $\c O_0$ in the variables $\{ e_u^* \}_{u \in M \setminus \{0\} }$, a set of formal symbols indexed by $M \setminus 0$. We endow $\c S_0^*( \c O_0)$ with the sup-norm on coefficients. Monomials in $\c S_0^*(\c O_0)$ have the form $e_{\b u}^* := e_{u_1}^* e_{u_2}^* \cdots e_{u_r}^*$, where $u_1, \ldots, u_r \in M(f) \setminus \{0\}$ for $r > 0$, and $e_0^* := 1$ when $r = 0$. Thus, elements in the ring may be described by
\[
\c S^*_0(\c O_0) := \left\{ \xi^* = \sum_{\b u \in \c S(M)} \xi^*(\b u) e_{\b u}^* \mid \xi^*(\b u) \in \c O_0  \right\}.
\]
Using the same notation as before, define the embedding $\Upsilon: \c C_0^*(\c O_0) \hookrightarrow \c S_0^*(\c O_0)$ by $\Upsilon(\tilde \pi^{-w(u)} x^{-u}) := e_u^*$ for $u \in M \setminus \{ 0\}$, and $\Upsilon(1) := e_0^* = 1$. For each $m \geq 1$, recall from Lemma \ref{L: better conv},  $\alpha_m^*(1)= 1 + \eta_m^*(\lambda, x)$ for some element $\eta_m^* \in \c C_0^*({\c O}_{0, p^m})$ satisfying $|\eta^*_m| < 1$. It follows that $\left(\Upsilon \circ \alpha_m^* (1) \right)^{\tau} \in \c S_0^*({\c O}_{0, p^m})$ for any $\tau \in \bb Z_p$. For $m \geq 1$,  we define the map $[\alpha_m^*]_\kappa: \c S_0^*(\c O_{0, p^m}) \rightarrow \c S_0^*({\c O}_{0, p^m})$ by 
\[
[\alpha_m^*]_\kappa( e_{u_1}^* \cdots e_{u_r}^* ) := (\Upsilon ( \alpha_m^*(1) ))^{\kappa - r} (\Upsilon ( \alpha_m^* (\tilde \pi^{-w(u_1)}\ x^{-u_1}) )) \cdots (\Upsilon ( \alpha_m^* (\tilde \pi^{-w(u_r)} x^{-u_r}))).
\]
The product on the right side makes sense and lives in $\c S_0^*(\c O_{0, p^m})$ since $\c S_0^*(\c O_{0, p^m})$ is a ring and each factor is clearly in $\c S_0^*(\c O_{0, p^m})$. Furthermore,
\begin{equation}\label{E: number14}
|[\alpha_m^*]_{\kappa}(e_{\b u}^*) | \leq |\tilde{\pi}^{(p-1)w(\b u)}|.
\end{equation}

Define the $R$ module
\[
\c O_{0, q}^* := \left\{ \zeta^* = \sum_{\gamma \in M(\Gamma)} \zeta^*(\gamma) \tilde \pi^{- w_{q\Gamma}(\gamma)} \lambda^{-\gamma} \mid \zeta^*(\gamma) \in R \right\}.
\]
Here  we do not insist that coefficients go to $0$ and we do not claim $\c O_{0,q}^* $ is a ring. As usual we define an absolute value on $\c O_{0, q}^*$ by $|\zeta^*| := \sup_{\gamma \in M(\Gamma)} | \zeta^*(\gamma)|$. 
For series in $\lambda$, we define a projection (or truncation) map
\[
\text{pr}_{M(\Gamma)}: \quad
\sum_{\gamma \in \bb Z^s} A(\gamma) \lambda^{-\gamma} \longmapsto \sum_{\gamma \in M(\Gamma)} A(\gamma) \lambda^{-\gamma}.
\]
Note that for any $q$ a power of the prime $p$, if $\gamma, \gamma', \textup{and} \ \delta$ all belong to $M(\Gamma)$ with $\gamma - \gamma' = -\delta$ then ${w_{q \Gamma}(\gamma) - w_{q \Gamma}(\gamma') \geq -w_{q \Gamma}(\delta)}$. It follows that for $\xi \in \c O_{0, q}$ and $\xi^* \in \c O_{0, q}^*$,   
\begin{equation}\label{E: proj}
pr_{M(\Gamma)}( \xi \cdot \xi^*) \in \c O_{0, q}^*.
\end{equation}
Define the $R$ module
\[
\c S_0^*(\c O_0^*) :=  \left\{ \omega^* = \sum_{\gamma \in M(\Gamma), \b u \in \c S(M)} \omega^*(\gamma, \b u) \tilde \pi^{-w_\Gamma(\gamma)}\lambda^{-\gamma} e_{\b u}^* \mid \omega^*(\gamma, \b u) \in R  \right\}
\]
Define the map $\Phi_\lambda$ by $\lambda \mapsto \lambda^p$. We define an $R$-linear map
\[
\beta_{\kappa, \bar t}^* := pr_{M(\Gamma)} \circ [ \alpha_{a d(\bar t)}^* ]_\kappa \circ \Phi_\lambda^{a d(\bar t)}
\]
by ``linearly'' extending over $R$ the action
\[
\beta_{\kappa, \bar t}^*( \lambda^{-\gamma} e_{\b u}^*) = pr_{M(\Gamma)}\left( \lambda^{-q_{\bar t} \gamma} \cdot [ \alpha_{a d(\bar t)}^* ]_\kappa( e_{\b u}^*) \right).
\]

\begin{lemma}\label{L: beta is endo}
$\beta_{\kappa, \bar t}^*$ is an $R$-linear endomorphism of $\c S_0^*(\c O_0^*)$. 
\end{lemma}

\begin{proof}
We have remarked already that $[\alpha_{a d(\bar t)}^*]_\kappa$ is a well-defined endomorphism of $\c S_0^*(\c O_{0, q_{\bar t}})$. As such, we may write for each $\b u \in \c S(M)$,
\[
[ \alpha_{a d(\bar t)}^* ]_\kappa( e_{\b u}^*) = \sum_{\sigma \in M(\Gamma), \b v \in S(M)} B_{\b u}(\sigma, \b v) \tilde \pi^{w_{q_{\bar t}\Gamma}(\sigma)} \lambda^\sigma e_{\b v}^* \in \c S_0^*(\c O_{0, q_{\bar t}}),
\]
with $B_{\b u}(\sigma, \b v) \in R$, $B_{\b u}(\sigma, \b v) \rightarrow 0$ as $w_{q_{\bar t}\Gamma}(\sigma) + w(\b v) \rightarrow \infty$ using (\ref{E: number14}). For $\omega^* = \sum_{\gamma \in M(\Gamma), \b u \in S(M)} \omega^*(\gamma, \b u) \tilde \pi^{-w_\Gamma(\gamma)} \lambda^{-\gamma} e_{\b u}^* \in S_0^*(\c O_0^*)$, we have 
\begin{align*}
\beta_{\kappa, \bar t}^*(\omega^*) &= \text{pr}_{M(\Gamma)}\left( \sum_{\gamma \in M(\Gamma), \b u \in S(M)} \omega^*(\gamma, u) \tilde{\pi}^{-w_{\Gamma}(\gamma)}\lambda^{- q_{\bar t} \gamma} \cdot [\alpha_{a d(\bar t)}^*]_\kappa (e_{\b u}^*) \right) \\
&= pr_{M(\Gamma)} \left( \sum_{\gamma \in M(\Gamma)} \lambda^{- q_{\bar t} \gamma} \sum_{ \b u \in S(M)} \omega^*(\gamma, \b u) \sum_{\sigma \in M(\Gamma), \b v \in S(M)} B_{\b u}(\sigma, \b v) \tilde \pi^{-w_{q_{\bar t} \Gamma}(\sigma)} \tilde{\pi}^{-w_{\Gamma}(\gamma)} \lambda^\sigma e_{\b v}^* \right) \\
&= \sum_{\tau \in M(\Gamma), \b v \in S(M)} C(\tau, \b v) \tilde \pi^{-w_\Gamma(\tau)} \lambda^{-\tau} e_{\b v}^*,
\end{align*}
where
\[
C(\tau, \b v) := \sum_{\b u \in S(M)} \sum_{\substack{\gamma, \sigma \in M(\Gamma) \\ q_{\bar t} \gamma - \sigma = \tau}} \omega^*(\gamma, \b u) B_{\b u}(\sigma, \b v) \tilde \pi^{-w_\Gamma(\gamma) + w_{q_{\bar t}\Gamma}(\sigma) + w_\Gamma(\tau)}.
\]
Observe that the exponent of $\tilde \pi$ satisfies
\[
-w_\Gamma(\gamma) + w_{q_{\bar t}\Gamma}(\sigma) + w_\Gamma(\tau) \geq \left( 1 - \frac{1}{q_{\bar t}} \right) w_\Gamma(\tau),
\]
so that the term $\tilde \pi^{- w_\Gamma(\gamma) + w_{q_{\bar t} \Gamma}(\sigma) + w_\Gamma(\tau)}$ is bounded in norm by 1 since $w(\tau) \geq 0$, and  $\omega^*(\gamma, \b u) \ \text{and} \  B_{\b u}(\sigma,v) \in R$. On the other hand, $B_{\b u}(\sigma, \b v) \rightarrow 0$ as $w_\Gamma(\sigma) + w(\b v) \rightarrow \infty$ so that the coefficient $C(\tau, \b v)$ is defined, in $R$, and $\beta_\kappa^*( \omega^*) \in \c S_0^*(\c O_0^*)$. Clearly it is $R$-linear. 
\end{proof}

\bigskip\noindent {\bf Estimation using finite symmetric powers.} It is useful to estimate $\beta_{\kappa, \bar t}$ and $\beta_{\kappa, \bar t}^*$ using finite symmetric powers. For monomials $e_{\b u}$ or $e_{\b u}^*$, with $\b u \in S(M)$, $\b u = (u_1, \ldots, u_r) \in ( M(f) \setminus 0)^r$, we say  as usual that the degree or length of $e_{\b u}$ or $e_{\b u}^*$ is $r$. For $\xi \in \c S_0(\c O_0)$, define $\text{length}(\xi)$ as the supremum of the lengths of the monomials $e_{\b u}$ in the support of $\xi$ (i.e. those terms appearing with non-zero coefficients). In the case $\text{length}(\xi) = r$, we may write $\xi = \sum_{|\b u| \leq r} \xi(\b u) e_{\b u}$, and $\xi$ may be a series (not a polynomial), since $M(f)$ and the set of monomials of degee $\leq r$ are infinite in general.  Similarly for $\xi_{\b u}^*$. 

Let $k$ be a positive integer. Define $\c S_0^{(k)}(\c O_0) := \{ \xi \in \c S_0(\c O_0) \mid \text{length}(\xi) \leq k \}$. Then the map 
\[
E_0^{k-r} E_{u_1} \cdots E_{u_r} \longmapsto e_{u_1} e_{u_2} \cdots e_{u_r}
\]
identifies $Sym^{k} \c C_0(\c O_0)$ with $\c S_0^{(k)}(\c O_0)$ as $\c O_0$-submodules in $\c S_0(\c O_0)$. Similarly, we identify $Sym^k \c C_0^*(\c O_0)$ in $\c S_0^*(\c O_0)$ as the $\c O_0$-submodule $\c S_0^{*(k)}(\c O_0)$ of power series in $\{ e_{\b u}^* \mid | \b u| \leq k \}$ with coefficients in $\c O_0$. By transfer of structure, we have a pairing $(\cdot, \cdot)_k: \c S_0^{(k)}(\c O_0) \times \c S_0^{*(k)}(\c O_0) \rightarrow \c O_0$.

We now work over $R$ and define a new pairing $\langle \cdot, \cdot \rangle_k: \c S_0^{(k)}(\c O_0) \times \c S_0^{*(k)}(\c O_0^*) \rightarrow \Omega$ as follows. (Here again $\c S_0^{*(k)}(\c O_0^*)$ is the $R$-submodule of $\c S_0^*(\c O_0^*)$ of series with support in monomials of degree $\leq k$, namely $\{ e_{\b u}^* \mid |\b u| \leq k \}$, with coefficients in $\c O_0^*$.) Let $\xi := \sum_{\gamma \in M(\Gamma), \b u \in S(M)} \xi(\gamma, \b u) \tilde \pi^{w_\Gamma(\gamma)} \lambda^\gamma e_{\b u} \in  \c S_0^{(k)}(\c O_0)$, and $\xi^* := \sum_{\sigma \in M(\Gamma), \b v \in S(M)} \xi^*(\sigma, \b v) \tilde \pi^{-w_\Gamma(\sigma)} \lambda^{-\sigma} e_{\b v}^* \in \c S_0^{*(k)}(\c O_0^*)$, set
\[
\langle \xi, \xi^* \rangle_k := \sum_{\gamma \in M(\Gamma), \b u \in S(M)} \xi(\gamma, \b u) \xi^*(\gamma, \b u) (e_{\b u}, e_{\b u}^*)_k,
\]
where $(\cdot, \cdot)_k$ was defined above. (Observe that as defined, a denominator $k!$ is introduced, so $(e_{\b u}, e_{\b u}^*)_k$ is a rational number with $p$-adic valuation bounded below by $-k/(p-1)$. This is independent of $\b u$, so $\langle \xi, \xi^* \rangle_k$ is well-defined and takes values in the $R$-submodule of $\Omega$ consisting of elements with $ord_p c \geq -k / (p-1)$.) It is useful to think of $\langle \xi, \xi^* \rangle_k$ as the constant term with respect to  $\lambda$ and the $e_{\b u}$ and $e_{\b u}^*$ of the product $\xi \cdot \xi^*$, where the product $e_{\b u} \cdot e_{\b v}^*$ is defined to be zero if $\b u \not= \b v$, and $(e_{\b u}, e_{\b u}^*)_k$ if $\b u = \b v$.

Let $k_m$ be a sequence of positive integers which tend to infinity (in the usual archimedean sense) and such that $\lim_{m \rightarrow \infty} k_m = \kappa$ $p$-adically. For each $m$ we have a Frobenius map $Sym^{k_m} (\alpha_{a d(\bar t)})$ on $Sym^{k_m} \c C_o(\c O_0)$, as well as a Frobenius map $Sym^{k_m}( \alpha^*_{a d(\bar t)})$ on $Sym^{k_m} \c C_0^*(\c O_{0, q_{\bar t}})$. By transport of structure, we have then a Frobenius map $[ \alpha_{a d(\bar t)} ]_{(\kappa; m)}$ on $\c S_0^{(k_m)}(\c O_0)$ and a dual  Frobenius $[\alpha_{a d(\bar t)}^* ]_{(\kappa; m)}$ on $\c S_0^*(\c O_{0, q_{\bar t}})$. We extend by zero these maps to all of $\c S_0(\c O_0)$ and $\c S_0^*(\c O_{0, q_{\bar t}})$, respectively. That is, we define 
\[
[ \alpha_{a d(\bar t)} ]_{(\kappa; m)}( e_{\b u}) := 
\begin{cases}
[ \alpha_{a d(\bar t)} ]_{k_m}( e_{\b u} ) & \text{if $|\b u| \leq k_m$} \\
0 & \text{otherwise.}
\end{cases} 
\]
To avoid any possible confusion, we note
\begin{align*}
[\alpha_{a d(\bar t)}]_{(\kappa; m)}( e_{u_1} \cdots e_{u_r}) &=  (\Upsilon \circ \alpha_{a d(\bar t)}(1))^{k_m - r} ( \Upsilon \circ \alpha_{a d(\bar t)}\tilde \pi^{w(u_1)} x^{u_1} ) \cdots ( \Upsilon \circ \alpha_{a d(\bar t)} \tilde \pi^{w(u_r)} x^{u_r}) ) \\
&\cong \left( Sym^{k_m} \alpha_{a d(\bar t)}  \right)( E_0^{k_m - r} E_{u_1} \cdots E_{u_r}),
\end{align*}
when $r \leq k_m$. Similarly
\[
[ \alpha_{a d(\bar t)}^* ]_{(\kappa; m)}( e_{\b u}^*) := 
\begin{cases}
[ \alpha_{a d(\bar t)}^* ]_{k_m}( e_{\b u}^* ) & \text{if $|\b u| \leq k_m$} \\
0 & \text{otherwise.}
\end{cases} 
\]

\begin{lemma}
$\lim_{m \rightarrow \infty} [ \alpha_{a d(\bar t)} ]_{(\kappa; m)} = [\alpha_{a d(\bar t)}]_\kappa$ as maps from $\c S_0(\c O_0) \rightarrow \c S_0(\c O_{0, q_{\bar t}})$.
\end{lemma}

\begin{proof}
Write
\begin{equation}\label{E: double star}
\left( [\alpha_{a d(\bar t)}]_{(\kappa; m)} - [\alpha_{a d(\bar t)}]_\kappa \right) (e_{u_1} e_{u_2} \cdots e_{u_r} ) = \left( \Upsilon( \alpha_{a d(\bar t)}(1))^{k_m - r} - \Upsilon( \alpha_{ a d(\bar t)}(1))^{\kappa - r} \right) ( \Upsilon( \alpha_{a d(\bar t)}(\tilde \pi^{w(u_1)} x^{u_1} ))) \cdots ( \Upsilon( \alpha_{a d(\bar t)}(\tilde \pi^{w(u_r)} x^{u_r} ))).
\end{equation}
If $r \leq k_m$, then the first factor on the right may itself be factored into
\[
-\Upsilon( \alpha_{a d(\bar t)}(1))^{\kappa - r} (1 - (\Upsilon(\alpha_{a d(\bar t)}(1))^{k_m - \kappa}).
\]
Writing $\kappa = k_m + p^{\tau(m)} \sigma_m$ (with $\tau(m) \rightarrow \infty$ and $\sigma_m \in \bb Z_p$) then
\[
| 1 - (\Upsilon(\alpha_{a d(\bar t)}(1))^{k_m - \kappa} | \leq | \tilde \pi^{\tau(m) + 1} |
\]
as in the proof of \cite[Lemma 2.2]{MR3249829}, and using the estimate (\ref{E: upsilon estimate}). If $r > k_m$ then (\ref{E: double star}) becomes
\begin{align*}
\left( [\alpha_{a d(\bar t)}]_{(\kappa; m)} - [\alpha_{a d(\bar t)}]_\kappa \right) (e_{\b u} ) = -[\alpha_{a d(\bar t)}]_\kappa e_{\b u} = - \Upsilon( \alpha_{ a d(\bar t)}(1))^{\kappa - r} ( \Upsilon( \alpha_{a d(\bar t)}(\tilde \pi^{w(u_1)} x^{u_1} ))) \cdots ( \Upsilon( \alpha_{a d(\bar t)}(\tilde \pi^{w(u_r)} x^{u_r} )))
\end{align*}
so that focussing on the $r$ rightmost factors,
\[
\left|   \left( [\alpha_{a d(\bar t)}]_{(\kappa; m)} - [\alpha_{a d(\bar t)}]_\kappa \right) e_{\b u} \right| \leq |\tilde \pi|^{\frac{p^{a d(\bar t) - 1} - 1}{p-1} \frac{1}{p^{a d(\bar t) - 1}} w(\b u)}
\]
coming from (\ref{E: upsilon estimate}). But  $w(\b u) \geq r w_0 > k_m w_0$ (where $w_0 := \min \{ w(\b u) \mid \b u \in M(f) \setminus \{0\} \}$). In terms of the operator norm,
\[
\|  [\alpha_{a d(\bar t)}]_\kappa - [\alpha_{a d(\bar t)}]_{(\kappa; m)} \| \leq | \tilde \pi |^{\min \left\{ \tau(m) + 1, \frac{p^{a d(\bar t) -1} - 1}{p-1} \frac{1}{p^{a d(\bar t) -1}} k_m w_0 \right\} }.
\]
As $k_m$ and $\tau(m)$ both tend to infinity as $m$ grows, we see that  $\lim_{m \rightarrow \infty} [ \alpha_{a d(\bar t)} ]_{(\kappa; m)} = [\alpha_{a d(\bar t)}]_\kappa$.
\end{proof}

In an altogether similar manner, we have by Lemma \ref{L: better conv}, for $u \not= 0$, $\alpha_m^*(\tilde \pi^{-w(u)} x^{-u})$ belongs to $\c C_0^*(\c O_{0, p^m})$, and (recalling (\ref{E: an estimate}))
\[
| \alpha_m^*(\tilde \pi^{-w(u)} x^{-u}) | \leq | \tilde \pi |^{(p-1) w(u)}.
\]
Also $\alpha_m^*(1) = 1 + \eta^*(\lambda)$ with $\eta^*(\lambda) \in \c O_{0, p^m}$ and $|\eta^*(\lambda)| \leq |\tilde \pi|$. With these observations, an entirely similar argument shows 
$\lim_{m \rightarrow \infty} [ \alpha_{a d(\bar t)}^* ]_{(\kappa; m)} = [\alpha_{a d(\bar t)}^*]_\kappa$ as maps from $\c S_0^*(\c O_{0, q_{\bar t}}) \rightarrow \c S_0^*(\c O_{0, q_{\bar t}}) $. Define
\begin{align*}
\beta_{(\kappa; m), \bar t} &:= \psi_\lambda^{a d(\bar t)} \circ [ \alpha_{a d(\bar t)}]_{(\kappa; m)} \\
\beta_{(\kappa; m), \bar t}^* &:=  \text{pr}_{M(\Gamma)} \circ [ \alpha_{a d(\bar t)}^*]_{(\kappa; m)} \circ \Phi_\lambda^{a d(\bar t)}.
\end{align*}
As $\psi_\lambda$ and $\Phi_\lambda$ are bounded maps, it follows that as operators on $\c S_0(\c O_0)$ and $\c S_0^*(\c O_0^*)$, respectively, 
\begin{align}\label{E: beta limit}
\lim_{m \rightarrow \infty} \beta_{(\kappa; m), \bar t} &= \beta_{\kappa, \bar t} \\
\lim_{m \rightarrow \infty} \beta_{(\kappa; m), \bar t}^* &= \beta_{\kappa, \bar t}^*. \notag
\end{align}

\begin{lemma}
For $\xi \in \c S_0^{(k_m)}(\c O_0)$ and $\xi^*  \in \c S_0^{*(k_m)}(\c O_0^*)$, 
\begin{equation}\label{E: pairing with estimate}
\langle \beta_{(\kappa; m), \bar t} \xi, \xi^* \rangle_{k_m} = \langle \xi, \beta_{(\kappa; m), \bar t}^* \xi^* \rangle_{k_m}.
\end{equation}
\end{lemma}

\begin{proof}
With $\xi \in \c S_0^{(k_m)}(\c O_0)$ and $\xi^* \in \c S_0^{*(k_m)}(\c O_{0, q_{\bar t}})$, we may rewrite (\ref{E: dual sym}) as
\begin{equation}\label{E: rewrite duality}
( [ \alpha_{a d(\bar t)}]_{(\kappa, m)} \xi, \xi^*)_{k_m} = (\xi,  [\alpha_{a d(\bar t)}^*]_{(\kappa; m)} \xi^*)_{k_m}.
\end{equation}
Next, by linearity we only need consider $\xi = \lambda^\gamma e_{\b u}$ and $\xi^* = \lambda^{-\sigma} e_{\b v}^*$ where $\gamma, \sigma \in M(\Gamma)$ and $\b u, \b v \in S(M)$. We may write
\[
(e_{\b u}, [\alpha_{a d(\bar t)}^*]_{(\kappa; m)} e_{\b v}^*)_{k_m} = \sum_{\tau \in M(\Gamma)} C(\tau) \lambda^\tau.
\]
Next, observe that
\[
\langle \psi_\lambda \xi, \xi^* \rangle_{k_m} = \langle \xi, \Phi_\lambda \xi^* \rangle_{k_m}.
\]
Hence, in the case $\xi = \lambda^\gamma e_{\b u}$ and $\xi^* = \lambda^{-\sigma} e_{\b v}^*$,
\begin{align*}
\langle \beta_{(\kappa; m)} \xi, \xi^* \rangle_{k_m} &= \langle [\alpha_{a d(\bar t)}]_{(\kappa; m)} \xi, \Phi_\lambda^{a d(\bar t)} \xi^* \rangle_{k_m} \\
&= \text{constant term of } \left[ \lambda^{\gamma - q_{\bar t} \sigma} ([\alpha_{a d(\bar t)}]_{(\kappa; m)} e_{\b u}, e_{\b v}^*)_{k_m} \right] \\
&= \text{constant of } \left[ \lambda^{\gamma - q_{\bar t} \sigma} ( e_{\b u}, [\alpha_{a d(\bar t)}^*]_{(\kappa; m)} e_{\b v}^*)_{k_m} \right] \qquad \text{ by (\ref{E: rewrite duality}}) \\
&= \text{constant of } \left[ \lambda^{\gamma - q_{\bar t} \sigma} \sum_{\tau \in M(\Gamma)} C(\tau) \lambda^{\tau} \right] \\
&=
\begin{cases}
C(q_{\bar t} \sigma - \gamma) & \text{if } q_{\bar t} \sigma - \gamma \in M(\Gamma) \\
0 & \text{otherwise.}
\end{cases}
\end{align*}
In the other direction, again setting $\xi = \lambda^\gamma e_{\b u}$ and $\xi^* = \lambda^{-\sigma} e_{\b v}^*$,
\begin{align*}
\langle \xi, \beta_{(\kappa; m)}^* \xi^* \rangle_{k_m} &= \text{constant term of }  \left[ \lambda^\gamma \cdot \text{pr}_{M(\Gamma)}\left( \lambda^{-q_{\bar t} \sigma} (e_{\b u}, [\alpha_{a d(\bar t}^*]_{(\kappa; m)} e_{\b v}^*)_{k_m} ) \right) \right] \\
&= \text{constant of } \left[ \lambda^\gamma \cdot \text{pr}_{M(\Gamma)}\left( \sum_{\tau \in M(\Gamma)} C(\tau) \lambda^{-(q_{\bar t} \sigma - \tau)} \right) \right] \\
&= \text{constant of } \left[ \lambda^\gamma \cdot  \sum_{\substack{\tau \in M(\Gamma) \text{ s.t.} \\ q_{\bar t} \sigma - \tau \in M(\Gamma)}} C(\tau) \lambda^{-(q_{\bar t} \sigma - \tau)} \right] \\
&= 
\begin{cases}
C(q_{\bar t} \sigma - \gamma) & \text{if } q_{\bar t} \sigma - \gamma \in M(\Gamma) \\
0 & \text{otherwise.}
\end{cases}
\end{align*}
\end{proof}

Observe that $\beta_{(\kappa; m), \bar t}$ and $\beta_{\kappa, \bar t}$ are completely continuous operators on the $p$-adic Banach $R$-algebra $\c S_0(\c O_0)$ (viewed as $R$-algebra) with orthonormal basis $\{ \tilde \pi^{w_\Gamma(\gamma)} \lambda^\gamma e_{\b u} \mid \gamma \in M(\Gamma), \b u \in S(M) \}$. Let $\c T_0(R) \ \text{be} \ \c S_0(\c O_0)$ viewed in this way as $R$-algebra. Similarly, write $\c T_0^*(R)$ for the $b(I)$-space (over $R$) in Serre's terminology with ``basis'' $I := \{ \tilde \pi^{-w_\Gamma(\gamma)} \lambda^{-\gamma} e_{\b u}^* \mid \gamma \in M(\Gamma), \b u \in S(M) \}$ with coefficients in $R$. Again, $\c T_0^*(R)$ is just $\c S_0^*(\c O_0^*)$ viewed over $R$. Then
\[
\lim_{m \rightarrow \infty} det(1 - \beta_{(\kappa; m), \bar t} T) = det(1 - \beta_{\kappa, \bar t} T).
\]

Similarly, $\beta^*_{(\kappa; m), \bar t}$ is a continuous $R$-linear endomorphism of $\c T_0^*(R)$ to itself. We may consider a matrix $\f B^{*(\kappa; m), \bar t}$ with entries in $R$ defined by
\[
\beta^*_{(\kappa; m), \bar t}( \tilde \pi^{-w_{\Gamma}(\gamma)} \lambda^{-\gamma} e_{\b u}^*) = \sum \f B^{*(\kappa; m), \bar t}_{(\delta, \b v), (\gamma, \b u)} \tilde \pi^{-w_\Gamma(\delta)} \lambda^{-\delta} e_{\b v}^*.
\]
Using the matrix $\f B^{*(\kappa; m), \bar t}$, we define in the usual way the Fredholm determinant $det(1 - \beta_{(\kappa; m), \bar t}^* T) = \sum_{j \geq 0} (-1)^{j+1} C_j(\beta_{(\kappa; m), \bar t}^*) T^j$ where $C_0  = 1$ and $C_j$ is the series of all principal $j \times j$ subdeterminants of the matrix $\f B^{*(\kappa; m), \bar t}$. The $\langle \cdot, \cdot \rangle_{k_m}$-adjointness of $\beta_{(\kappa; m), \bar t}$ and $\beta^*_{(\kappa; m), \bar t}$ implies $C_j(\beta_{(\kappa; m), \bar t}) = C_j(\beta_{(\kappa; m), \bar t}^*)$, so that
\[
det(1 - \beta_{(\kappa; m), \bar t}^* T) = det(1 - \beta_{(\kappa; m), \bar t} T).
\]
The uniform convergence $\lim_{m \rightarrow \infty} \f B^{*(\kappa; m), \bar t}  =: \f B^*_{\kappa, \bar t}$  over the entries implies that the series $\sum_{j \geq 0} (-1)^{j+1} C_j(\f B_{\kappa, \bar t}^*) T^j$ is well-defined, and is the coefficient-wise limit  of $det(1 - \f B^*_{(\kappa; m), \bar t} T)$ as $m \rightarrow \infty$. If we then define
\[
det(1 - \beta^*_{\kappa, \bar t} T) := \sum_{j \geq 0} (-1)^{j+1} C_j(\f B_{\kappa, \bar t}^*) T^j,
\]
then we have shown:

\begin{theorem}
$det(1 - \beta_{\kappa, \bar t} T) = det(1 - \beta_{\kappa, \bar t}^* T)$, and thus from (\ref{E: A}),
\begin{equation}\label{E: L0  and dual}
L^{(0)}(\kappa, \bar t, T)^{(-1)^{s+1}} = det(1 - \beta_{\kappa, \bar t}^* T)^{\delta_{q_{\bar t}}^{s}}.
\end{equation}
\end{theorem}

\section{Eigenvector}

Recall that $G(t, \lambda, x) = f(t, x) + P(\lambda, x) = \sum t_u x^u + \sum A(\gamma, v) \lambda^\gamma x^v \in \bb F_q[x_1^\pm, \ldots, x_n^\pm, \lambda_1^\pm, \ldots, \lambda_s^\pm, \{ t_u \}_{u \in Supp(f)}]$. Let  $\hat A(\gamma, v)$ be the Teichm\"uller lift in $\bb Q_q$ for each $(\gamma, v) \in Supp(P)$, and denote the lifting of $G$ by $\hat G(t, \lambda, x) := \hat f(t, x) + \hat P(\lambda, x) = \sum t_u x^u + \sum \hat A(\gamma, v) \lambda^\gamma x^v \in \bb Q_q[x_1^\pm, \ldots, x_n^\pm, \lambda_1^\pm, \ldots, \lambda_s^\pm, \{ t_u \}_{u \in Supp(f)}]$. We now replace every coefficient of $G$ (w.r.t the variables $x$ and $\lambda$) with a new variable $\Lambda$: set $ \b f(\Lambda,x)=\sum_{u \in Supp(f)} \Lambda_u x^u \ \text{and} \ \c P(\Lambda, \lambda. x) = \sum_{(\gamma, v) \in Supp(P)} \Lambda_{\gamma, v} \lambda^\gamma x^v$ and
\[
H(\Lambda, \lambda, x) := \b f(\Lambda,x)  +  \c P(\Lambda, \lambda. x) .
\]
As before, let $\Delta_{\infty}(H)$ denote the Newton polytope at infinity in $\bb R^{s+n}$ of $H$ (in $\lambda$ and $x$ variables). Let $Cone(H)$ be the cone in $\bb R^{s+n}$ over $\Delta_{\infty}(H)$ and $M(H) = Cone(H) \cap \bb Z^{s+n}$ be the relevant monoid. Clearly $M(H) \subset M(\Gamma) \times M(f)$. By our hypothesis that the $x$-support of $P$ is contained in $\Delta_{\infty}(f)$ we have the polyhedral weight function on this polytope $w_H$ dominates the total weight $w_{\text{tot}} := w_{\Gamma} + w_f$ relative to the polyhedron $\Gamma \times \Delta_{\infty}$; more precisely 
\[
w_{\text{tot}}(\gamma, u) \leq w_H(\gamma, u)
\]
for all $(\gamma, u) \in M(H)$.

Now recall as well the projection map defined earlier,
\[
\text{pr}_{M(f)}: \sum_{u \in \bb Z^n} C(u) x^{-u} \longmapsto \sum_{u \in M(f)} C(u) x^{-u}.
\]
We define $\c K := R[[\Lambda]]$ and 
\[
\c K_0 := \{ \xi \in \sum_{v \in \bb Z_{\geq 0}^t} \xi_v \Lambda^v \in \c K \mid \xi_v \rightarrow 0 \text{ as } v \rightarrow \infty \},
\]
where $t$ is the cardinality of $\{ Supp(f) \} \cup \{ Supp(P) \}$. 
We endow both spaces with the sup norm on coefficients.

Define the spaces 
\begin{align*}
\c W(\c K_0) &:= \{ \sum_{\gamma \in M(\Gamma)} \xi_\gamma(\Lambda) \tilde \pi^{-w_\Gamma(\gamma)} \lambda^{-\gamma} \mid \xi_\gamma(\Lambda) \in \c K_0 \} \\
\c W_0(\c K_0) &:= \{ \sum_{\gamma \in M(\Gamma)} \xi_\gamma(\Lambda) \tilde \pi^{-w_\Gamma(\gamma)} \lambda^{-\gamma} \mid \xi_\gamma(\Lambda) \in \c K_0, \xi_\gamma(\Lambda) \rightarrow 0 \text{ as } \gamma \rightarrow \infty \}.  
\end{align*}
Similarly, we define as well
\begin{align*}
\c D( \c W(\c K_0)) &:= \{ \sum_{(\gamma, u) \in M(\Gamma) \times M(f) } \xi_{\gamma, u}(\Lambda) \tilde \pi^{-w_\Gamma(\gamma) - w(u)} \lambda^{-\gamma} x^{-u} \mid \xi_{\gamma, u}(\Lambda) \in \c K_0 \} \\
\c D_0( \c W_0( \c K_0) ) &:= \{ \sum_{(\gamma, u) \in M(\Gamma) \times M(f) } \xi_{\gamma, u}(\Lambda) \tilde \pi^{-w_\Gamma(\gamma) - w(u)} \lambda^{-\gamma} x^{-u}  \mid \xi_{\gamma, u}(\Lambda) \in \c K_0,  \sup_{\gamma} | \xi_{\gamma, u}(\Lambda)| \rightarrow 0 \text{ as } u \rightarrow \infty \}.
\end{align*}
We proceed similar to our work above. We define a $\c K_0$-module
\[
\c S^*_0(\c W(\c K_0)) := \{ \sum_{\gamma \in M(\Gamma), \b u \in S(M)} A_{\gamma, \b u}(\Lambda) \tilde \pi^{-w_\Gamma(\gamma)} \lambda^{-\gamma} e_{\b u}^* \mid A_{\gamma, \b u} \in \c K_0 \}.
\]
and a $\c W_0(\c K_0)$-algebra
\[
\c S^*_0(\c W_0(\c K_0)) := \{ \sum_{\gamma \in M(\Gamma), \b u \in S(M)} \xi^*_{ \b u}(\Lambda, \lambda) e_{\b u}^* \mid \xi^*_{ \b u}(\Lambda, \lambda) \in \c W_0(K_0) \}.
\]
As before define an embedding $\Upsilon: \c D(\c W(\c K_0))) \hookrightarrow \c S^*_0(\c W(\c K_0))$ by $\tilde \pi^{-w(u)} x^{-u} \mapsto e_u^*$ for $u \in M \setminus \{0\}$ and $\Upsilon(1) := 1$. We define a relative Frobenius map as follows. First, set
\begin{align*}
F(\Lambda, \lambda, x) &:= \prod_{u \in Supp(f)} \theta(\Lambda_u x^u) \cdot \prod_{(\gamma, v) \in Supp(P)} \theta( \Lambda_{\gamma, v} \lambda^\gamma x^v) \\
F_m(\Lambda, \lambda, x) &:= \prod_{i=0}^{m-1} F(\Lambda^{p^i}, \lambda^{p^i}, x^{p^i}),
\end{align*}
and note that, similar to before,
\[
F_m(\Lambda, \lambda, x) = \sum_{(\gamma, u) \in M(H)} \c F^m_{\gamma, u}(\Lambda) \lambda^{\gamma} x^u 
\]
with $\c F^m_{\gamma, u}(\Lambda)= B_{\gamma,u}(\Lambda) \tilde{\pi}^{w_{tot}(\gamma,u)/p^{m-1}} = \b B_{\gamma,u}(\Lambda) \tilde{\pi}^{w_H(\gamma,u)/p^{m-1}}$.
It follows that, if we set (as before)
\[
\alpha_m^*(\Lambda, \lambda) := \text{pr}_{M(f)} \circ F_m(\Lambda, \lambda, x) \circ \Phi_x^m,
\]
where $\Phi_x$ sends $x^u \mapsto x^{pu}$ and $\text{pr}_{M(f)}$ was defined above, then an argument similar to Lemma \ref{L: better conv} shows
\[
\alpha_{m, \Lambda}^*: \c D_0(\c W_{0, p^m}(\c K_0)) \rightarrow \c D_0(\c W_{0, p^m}(\c K_0)),
\]
where $\c W_{0, p^m}$ is defined by replacing $\tilde \pi^{w_\Gamma}$ with $\tilde \pi^{w_{p^m \Gamma}}$ in the definition of $\c W_0$. Furthermore, 
\[
\alpha_{m, \Lambda}^*(\tilde \pi^{-w(v)} x^{-v}) = \sum_{u \in M(f)} C_{u, v}(\Lambda, \gamma) \tilde \pi^{-w(u)} x^{-u}
\]
with $C_{u, v}(\Lambda, \lambda) \in \c W_{0, p^m}(\c K_0)$ and $C_{u, v}(\Lambda, \lambda) \rightarrow 0$ as $w(u) \rightarrow \infty$. For any $\kappa \in \bb{Z}_p$, we define $[\alpha^*_m]_{\kappa} : \c S^*_{0,p^m}(\c W_0(\c K_0)) \rightarrow \c S^*_{0,p^m}(\c W_0(\c K_0))$ using (\ref{E: number14}). By an argument similar to Lemma \ref{L: beta is endo}, the map
\[
\beta_{\kappa, \bar t, \Lambda}^*: \c S^*_0(\c W(\c K_0)) \rightarrow \c S^*_0(\c W(\c K_0))
\]
defined by 
\[
\beta_{\kappa, \bar t, \Lambda}^* := \text{pr}_{M(\Gamma)} \circ [ \alpha_{a d(\bar t)}^* ]_\kappa  \circ \Phi_\lambda^{a d(\bar t)}.
\]
is an endomorphism over $\c K_0$.

\bigskip\noindent{\bf Eigenvector.} Set $M_0(\Gamma) = M(\Gamma) \cap (- M(\Gamma)), M_0(f) = M(f) \cap (- M(f)), \ \text{and} \ M_0(H) = M(H) \cap (-M(H))$. Define the projection map 
\[
\text{pr}_0: \quad  \sum_{\gamma \in \bb Z^s, u \in \bb Z^n} C(\gamma, u) \lambda^\gamma x^u \longmapsto \sum_{\gamma \in M_0(\Gamma), u \in M_0(f)} C(\gamma, u) \lambda^\gamma x^u.
\]
If we write $\exp \pi H(\Lambda, \lambda, x) = \sum A_{v, \gamma, u} \Lambda^v \lambda^\gamma x^u$ then clearly $A_{v, \gamma, u} \in R$. Let us write then $\exp \pi H(\Lambda, \lambda, x) = \sum A_{\gamma, u}(\Lambda) \lambda^\gamma x^u$ with $A_{\gamma, u}(\Lambda) \in R[[\Lambda]]$ and the indices $(\gamma, u) \in M(H) \subset M(\Gamma) \times M(f)$. We will also write 
\[
pr_0( \exp \pi H(\Lambda, \lambda, x) ) = \sum_{(\gamma, u) \in M_0(\Gamma) \times M_0(f)} J_{\gamma, u}(\Lambda) \tilde \pi^{-w_\Gamma(\gamma) -w(u)} \lambda^{-\gamma} x^{-u} =  \sum_{(\gamma, u) \in M_0(H)} \tilde{J}_{\gamma, u}(\Lambda) \tilde \pi^{-w_H(\gamma)} \lambda^{-\gamma} x^{-u}.
\]
The running indices $(\gamma, u)$ in all these sums may be taken in $M_0(H).$ Of course, 
\begin{equation}\label{E: J's}
J_{\gamma, u}(\Lambda) = A_{-\gamma, -u}(\Lambda) \tilde \pi^{w_\Gamma(\gamma) + w(u)} = \tilde{J}_{\gamma,u} \tilde{\pi}^{ w_{\Gamma}(\gamma)+w(u) -w_H(\gamma,u) }
\end{equation}
for every $(\gamma, u) \in M_0(H)$, and $J_{0,0} = \tilde{J}_{0,0} = A_{0,0}  \in 1 + \Lambda \c K$.  That is, $J_{0,0}(\Lambda)$ is a power series in the variables $ \Lambda$ with coefficients in $R$ and constant term 1. So $J_{0,0}(\Lambda)$ is a unit in $\c K$. Define 
\begin{align}\label{E: eigenvector}
\eta(\Lambda, \lambda, x) :&= \frac{1}{ J_{00}(\Lambda) } \text{pr}_0 \exp \pi H(\Lambda, \lambda, x) \notag \\
&= 1 + \sum_{\substack{ (\gamma, u) \in M_0(H)  \\  (\gamma, u) \not= (0,0)}} \frac{\tilde{J}_{\gamma, u}(\Lambda)}{J_{0,0}(\Lambda)} \tilde \pi^{-w_H(\gamma,u)} \lambda^{-\gamma} x^{-u}. 
\end{align}
In \cite{MR2966711}, it was shown that $J_{0,0}(\Lambda) / J_{0,0}(\Lambda^p)$ and $\tilde{J}_{\gamma, u}(\Lambda) / J_{0,0}(\Lambda)$ converge on the closed unit polydisk $| \Lambda | \leq 1$ for every $\Lambda$. Equivalently, $J_{0,0}(\Lambda) / J_{0,0}(\Lambda^p)$ and $ \tilde{J}_{\gamma, u}(\Lambda) / J_{0,0}(\Lambda)$ belong to $\c K_0$. The same holds as well for $ J_{\gamma, u}(\Lambda) / J_{0,0}(\Lambda)$ using (\ref{E: J's}), since $A_{-\gamma,-u}(\Lambda) \in R[[\Lambda]]$.

Since $J_{\gamma, u}(\Lambda) = A_{-\gamma, -u} \tilde \pi^{w_\Gamma(\gamma) + w(u)}$, we have
\[
\left| \sum_{\gamma \in M_0(\Gamma)} \frac{J_{\gamma, u}(\Lambda)}{ J_{0,0}(\Lambda)} \tilde \pi^{-w_\Gamma(\gamma)} \lambda^{-\gamma} \right| \leq |\tilde \pi |^{w(u)},
\]
and so $\eta(\Lambda, \lambda, x) \in \c D_0(\c W_0(\c K_0)) \subset \c D( \c W(\c K_0))$.

Set $\c F(\Lambda) := J_{00}(\Lambda) / J_{00}(\Lambda^p)$. Observe that 
 \[\text{pr}_{M(f)} \circ F_m(\Lambda,\lambda,x) \circ \text{pr}_0 (\exp \pi H(\Lambda^p,\lambda^p,x^p)) = \text{pr}_{M(f)}  (\exp \pi H(\Lambda,\lambda,x))
\]
so that
\begin{align*}
\alpha_{1, \Lambda}^*( \eta(\Lambda^p, \lambda^p, x) ) &= \c F(\Lambda) \text{pr}_{M(f)} \left( \frac{\exp \pi H(\Lambda, \lambda, x)}{J_{00}(\Lambda)} \right) \\
&= \c F(\Lambda) \left( \eta(\Lambda, \lambda, x)  + \tilde \omega (\Lambda, \lambda, x) \right),
\end{align*}
where each $\lambda^\gamma$ appearing in $\tilde \omega$ lies in $M(\Gamma) \setminus M_0(\Gamma)$.

Iterating this, if we set
\[
\c F_m(\Lambda) := \prod_{i=0}^{m-1} \c F(\Lambda^{p^i}),
\]
then we have
\begin{equation}\label{E: first formula}
\alpha_{a d(\bar t), \Lambda}^* \eta(\Lambda^{q_{\bar t}}, \lambda^{q_{\bar t}}, x) = \c F_{a d(\bar t)}(\Lambda) \left( \eta(\Lambda, \lambda, x) + \omega(\Lambda, \lambda, x) \right),
\end{equation}
where each $\lambda^\gamma$ appearing in $\omega$ lies in $M(\Gamma) \setminus M_0(\Gamma)$. 

For notational convenience, set $Q_{\gamma, u} := Q_{\gamma, u}(\Lambda) := J_{\gamma, u}(\Lambda) / J_{0,0}(\Lambda)$ so that
\[
\eta(\Lambda, \lambda, x) = 1 + \sum_{\substack{\gamma \in M_0(\Gamma), u \in M_0(f) \\ (\gamma, u) \not= (0, 0)}} Q_{\gamma, u} \tilde \pi^{-w_\Gamma(\gamma) - w(u)} \lambda^{-\gamma}  x^{-u}.
\]
Next, write
\[
\Upsilon(\eta) = 1 + \sum_{\substack{\gamma \in M_0(\Gamma), u \in M_0(f) \\ (\gamma, u) \not= (0, 0)}} Q_{\gamma, u} \tilde \pi^{-w_\Gamma(\gamma)} \lambda^{-\gamma} e_u^*,
\]
and observe that $\Upsilon(\eta)$ and $1 / \Upsilon(\eta)$ are elements of $\c S_0^*(\c W_0(\c K_0))$. For $\kappa \in \bb Z_p$, we compute
\begin{align*}
(\Upsilon(\eta))^\kappa &= \sum_{l=0}^\infty \binom{\kappa}{l}(\sum Q_{\gamma, u} \tilde \pi^{-w_\Gamma(\gamma)} \lambda^{-\gamma} e_u^*)^l \\
&= \sum_{l=0}^\kappa \binom{\kappa}{l} \sum_{\substack{\gamma_1, \ldots, \gamma_l \in M_0(\Gamma) \\ u_1, \ldots, u_l \in M_0(f) \\ (\gamma_j, u_j) \not= (0, 0) \text{ for every j}} } Q_{\gamma_1, u_1} \cdots Q_{\gamma_l, u_l} \tilde \pi^{-w_\Gamma(\gamma_1) - \cdots - w_\Gamma(\gamma_l)} \lambda^{-(\gamma_1 + \cdots + \gamma_l)} e_{u_1}^* \cdots e_{u_l}^* \\
&= \sum_{l = 0}^\infty \binom{\kappa}{l} \sum_{\substack{\gamma_1, \ldots, \gamma_l \in M_0(\Gamma) \\ u_1, \ldots, u_l \in M_0(f) \\ (\gamma_j, u_j) \not= (0, 0) \text{ for every j}} } \widetilde{Q}_{\gamma, u} \cdot \tilde \pi^{-w_\Gamma(\gamma_1 + \cdots + \gamma_l)} \lambda^{-(\gamma_1 + \cdots + \gamma_l)} e_{u_1}^* \cdots e_{u_l}^*
\end{align*}
where
\[
\widetilde{Q}_{\gamma, u} := Q_{\gamma_1, u_1} \cdots Q_{\gamma_l, u_l} \tilde \pi^{-w_\Gamma(\gamma_1) - \cdots - w_\Gamma(\gamma_l) + w_\Gamma(\gamma_1 + \cdots + \gamma_l)}.
\]
Hence, $(\Upsilon(\eta))^\kappa \in \c S^*(\c W_0(\c K_0))$.  As every $\lambda^\gamma$ appearing in $\Upsilon(\omega)$ (from equation (\ref{E: first formula})) satisfies $\gamma \in M(\Gamma) \setminus M_0(\Gamma)$, it follows that the same is true for $( \Upsilon(\omega) / \Upsilon(\eta) )^r$ for any $r \in \bb Z_{\geq 1}$. Hence,
\[
\text{pr}_{M(\Gamma)} \left(1 + \frac{\Upsilon(\omega)}{\Upsilon(\eta) } \right)^\kappa  = 1.
\]

\bigskip\noindent{\bf Unit root formula.} We may now finish the proof of Theorem \ref{T: Main Thm}. For convenience, write $\eta(\Lambda, \lambda, x) = 1 + h(\Lambda, \lambda, x)$ so that $\Upsilon(\eta)^\kappa = (1 + \Upsilon(h))^\kappa = \sum_{l=0}^\infty  \binom{\kappa}{l} \Upsilon(h)^l$. Observe that 
\begin{align*}
\beta_{\kappa, \bar t, \Lambda}^* \Upsilon(\eta(\Lambda^{q_{\bar t}}, \lambda, x))^\kappa &= \text{pr}_{M(\Gamma)} \circ [\alpha_{a d(\bar t), \Lambda}^*]_\kappa \circ \Phi_\lambda^{a d(\bar t)} \Upsilon(\eta(\Lambda^{q_{\bar t}}, \lambda, x))^\kappa \\
&= \text{pr}_{M(\Gamma)} \circ [\alpha_{a d(\bar t), \Lambda}^*]_\kappa \Upsilon(\eta(\Lambda^{q_{\bar t}}, \lambda^{q_{\bar t}}, x))^\kappa \\
&= \text{pr}_{M(\Gamma)} \circ [\alpha_{a d(\bar t), \Lambda}^*]_\kappa  \sum_{l=0}^\infty \binom{\kappa}{l} \Upsilon \left( h(\Lambda^{q_{\bar t}}, \lambda^{q_{\bar t}}, x)  \right)^l \\
&= \text{pr}_{M(\Gamma)} \sum_{l=0}^\infty \binom{\kappa}{l} \left(\Upsilon \circ \alpha_{a d(\bar t), \Lambda}^* \cdot 1 \right)^{\kappa - l} \left(\Upsilon \circ  \alpha_{a d(\bar t), \Lambda}^* h(\Lambda^{q_{\bar t}}, \lambda^{q_{\bar t}}, x) \right)^l \qquad \text{by definition of $ [\alpha_{a d(\bar t), \Lambda}^*]_\kappa$} \\
&= \text{pr}_{M(\Gamma)} \> \left( \Upsilon \circ  \alpha_{a d(\bar t), \Lambda}^* \cdot 1 + \Upsilon \circ  \alpha_{a d(\bar t), \Lambda}^* h(\Lambda^{q_{\bar t}}, \lambda^{q_{\bar t}}, x) \right)^\kappa \\
&= \text{pr}_{M(\Gamma)} \>  \left( \Upsilon \circ  \alpha_{a d(\bar t), \Lambda}^* \eta(\Lambda^{q_{\bar t}}, \lambda^{q_{\bar t}}, x) \right)^\kappa \\
&= \text{pr}_{M(\Gamma)} \>  \c F_{a d(\bar t)}(\Lambda)^\kappa \left( \Upsilon(\eta(\Lambda, \lambda, x) + \Upsilon(\omega(\Lambda, \lambda, x)  \right)^\kappa \qquad \text{by } (\ref{E: first formula}) \\
&= \text{pr}_{M(\Gamma)} \>  \c F_{a d(\bar t)}(\Lambda)^\kappa \Upsilon(\eta(\Lambda, \lambda, x))^\kappa \left( 1 + \frac{\Upsilon(\omega(\Lambda, \lambda, x))}{\Upsilon( \eta(\Lambda, \lambda, x))}  \right)^\kappa \\
&= \c F_{a d(\bar t)}(\Lambda)^\kappa \Upsilon(\eta(\Lambda, \lambda, x))^\kappa.
\end{align*}
Finally, we may specialize this equality taking  $\Lambda$ at the Teichm\"uller unit coefficients of $\hat G(\hat t, \lambda, x)$:
\[
\Lambda_u = \hat t_u \qquad \text{and} \qquad \Lambda_{\gamma, v} = \hat A(\gamma, v) \qquad \text{for all $u$ and $\gamma, v$ in the support of $H$},
\] 
and setting
\[
\eta_{\text{sp}}(\lambda, x) := \left( \eta(\Lambda, \lambda, x) \text{ specialized at $\Lambda_u = \hat t_u$ and $\Lambda_{\gamma, v} = \hat A(\gamma, v)$} \right),
\]
then we see that
\begin{equation}\label{E: unit root dual}
\beta_{\kappa, \bar t}^* \Upsilon(\eta_{\text{sp}}(\lambda, x))^\kappa = \c F_{a d(\bar t)}(\hat t)^\kappa \Upsilon(\eta_{\text{sp}}(\lambda, x))^\kappa 
\end{equation}
This demonstrates that $\c F_{a d(\bar t)}(\hat t)^\kappa$ is the unique unit root of $L^{(0)}(\kappa, \bar t, T)^{(-1)^{s+1}}$ by (\ref{E: L0  and dual}), which, together with Theorem \ref{T: unique unit}, completes the proof of Theorem \ref{T: Main Thm}.

\bibliographystyle{amsplain}
\bibliography{References.bib}

\end{document}